\documentclass[12pt]{article}

\usepackage{amsmath,amsthm,amssymb}
\allowdisplaybreaks
\usepackage{mathrsfs}
\usepackage{cite}

\makeatletter\@addtoreset{equation}{section}\makeatother

\newtheorem{theorem}{Theorem}[section]

\newtheorem{corollary}{Corollary}[section]

\theoremstyle{remark}

\theoremstyle{remark}

\theoremstyle{remark}
\newtheorem{remark}{Remark}[section]

\newcommand{\curl}{\operatorname{curl}}
\newcommand{\grad}{\operatorname{grad}}
\newcommand{\supp}{\operatorname{supp}}
\newcommand{\be}[1]{\begin{equation}\label{#1}}
\renewcommand{\phi}{\varphi}
\newcommand{\diver}{\operatorname{div}}

\newcommand{\di}{\partial}

\begin{document}

\begin{center}{\Large \bf  Construction of  solutions to parabolic and hyperbolic  initial-boundary value problems}
\end{center}

{\large William G. Litvinov}\\
Institute of Mathematics, University of Augsburg, Universit\"atsstr. 14,\\
D-86159 Augsburg, Germany \\
e-mail: \texttt{william.litvinov@}gmail.com\vspace{2mm}

{\large Eugene Lytvynov}\\ Department of Mathematics,
Swansea University, Singleton Park, Swansea SA2 8PP, U.K.\\
e-mail: \texttt{e.lytvynov@swansea.ac.uk}\vspace{2mm}

{\small
\begin{center}
{\bf Abstract}
\end{center}
\noindent

We show that infinitely differentiable solutions to  parabolic and  hyperbolic equations, whose right-hand sides are 
analytical in time,  are also analytical in time  at each fixed point  of the space. These solutions are given in the form of the Taylor expansion with respect to time $t$ with coefficients depending on $x$. The coefficients of the expansion are defined by recursion relations, which are obtained  from the condition of compatibility of order $k=\infty$.  The value of the solution on the boundary is defined by the right-hand side and initial data, so that it is not prescribed. We show that exact  regular  and weak solutions to the initial-boundary value problems for parabolic and hyperbolic equations can be determined as the sum of a function that satisfies the boundary conditions and the limit of the infinitely differentiable solutions for smooth approximations of the data of the  corresponding   problem with zero boundary conditions.  These solutions  are represented in the form of the Taylor expansion with respect  to $t$. The suggested method can be considered as an alternative to  numerical methods of   solution of parabolic and hyperbolic equations.\vspace{2mm}

{\bf Key words:} Parabolic equation, hyperbolic equation, smooth solution, regular solution,
Taylor expansion.

\section{Introduction}

Initial-boundary value (mixed) problems for  parabolic  and hyperbolic equations have
since long ago led to a great number of works; see e.g.\ the monographs
\cite{Eid.,Fr.,LSU.,  LiM.1,   Sol.2}  and the references therein.

This paper is devoted to construction of infinitely differentiable solutions to
  parabolic  and hyperbolic equations, and its applications to construction
 of regular and weak  solutions to initial-boundary problems   for  these
  equations.

It is well known that, for the  existence of a smooth solution to  parabolic  or  hyperbolic
equation, the compatibility condition of an order $k\in \mathbb{N}$, corresponding  to
the smoothness of the solution  to the problem, should be satisfied.

The compatibility condition of order $k$  means that  the functions
$\frac{\di^i u}{\di t^i}\big|_{t=0}$, $i = 0,1,2,\dots,k$ ($u$  being the solution, $t$ time),
which are determined from the equation, initial data, and the right-hand side, should  be equal
on the boundary to $\frac{\di^i u_b}{\di t^i}\big|_{t=0}$, $i = 0,1,2,\dots,k$,
where $u_b$ is the given function of values of the solution on the boundary. In the case where the solution
is  infinitely differentiable, one has $k = \infty$.

We consider problems in a bounded domain $\Omega$ in ${\mathbb{R}}^n$ with a boundary $S$  of the
$C^{\infty}$ class on the time interval  $(0,T)$, $T<\infty$.

We  suppose that the coefficients of the equation, the right-hand side,
and the initial data are  infinitely
differentiable, and furthermore   the coefficients of the equation and  the right--hand side
are given in the form of the Taylor
expansion with respect to time $t$ with the origin at the point $t=0$ and with coefficients depending
on $x$,  where $x$ is a point in the space. Then the solution to the problem under consideration is informally
given in the form of the Taylor expansion with respect to $t$ in which  coefficients depend on $x$, i.e.,
\begin{equation}\label{1.1}
u(x,t) = \sum_{i=0}^{\infty} \, \frac 1{i!}\, \frac{\di^i u}{\di t^i}\, (x,0)t^i.
\end{equation}
The coefficients $\frac{\di^i u}{\di t^i}\, (\cdot,0)$ are determined by recurrence relations, more exactly,
they
are determined by the derivatives with respect to
time $t$  at $t=0$ of the right-hand side $f$,
the coefficients of the equation,
  and by the initial data $u_0$ for a parabolic equation and $u_0, u_1$ for a hyperbolic equation.

We prove  converges of the series \eqref{1.1}  in   the space  $C^{\infty}(\overline{Q})$, $Q=\Omega \times (0,T) $
 by using the existence of an infinitely
 differentiable solution to the problem.
 So that, the value of the solution $u$ on the boundary
 $u\big|_{S\times(0,T)} = u_b$ is uniquely  determined by $f$ and $u_0$  for a parabolic equation, and
 by  $f$, $u_0$ and  $u_1$ for a hyperbolic equation.

 This peculiarity is for the first time shown in our work. In the usual, accepted approach, one
 prescribes for parabolic and hyperbolic equations a right-hand side,
  initial, and  boundary
 conditions.

 For the zero Dirichlet boundary condition, we assume that $u_0$ and $u_1$ are elements of
 $\mathcal{D}(\Omega)$ and $f \in C^{\infty}([0,T];\mathcal{D}(\Omega))$. Then the compatibility
 condition of order $k= \infty$ is satisfied, and the solution to parabolic and hyperbolic  equations can be   represented in the form of \eqref{1.1}.

It is known that  the space  $C^{\infty}(\overline{Q}) $ is dense both in $W_q^l(Q)$
and in the space
 $(W_g^l(Q))^*$, $1/q +1/g  =1$,      that is the dual of $W_q^l(Q)$ for any $l\in \mathbb{N}$, $q\ge 2$.
 By  the corollary to the Weierstrass--Stone theorem, the set of products of polynomials with
   respect to $x$  and polynomials with respect to $t$ is dense in  $C^{\infty}(\overline{Q})$.
 Therefore, the set of functions
 that are represented in the form of the Taylor expansion with respect to $t$ with coefficients
 which are elements of the space $C^{\infty}(\overline{\Omega})$, is dense in $C^{\infty}(\overline{Q})$,
 in $W_q^l(Q)$, and in  $(W_g^l(Q))^*$ .

 Because of these properties, one can approximate smooth  and  non-smooth data of the problem
  and the coefficients of equation
  by corresponding infinitely differentiable functions with an arbitrary accuracy.

We apply the Taylor representation \eqref{1.1} to construction of regular and weak solutions
 to  parabolic and hyperbolic equations for which we prescribe the right-hand side, initial,
 and boundary conditions. We consider well-posed parabolic and hyperbolic problems for which
 the solution depends  continuously on the data of the problem.
 The problems with
 inhomogeneous boundary conditions are reduced to problems with homogeneous boundary
 conditions. The data of these problems are approximated by corresponding infinitely differentiable
 functions for which the compatibility condition of order $k= \infty$ is satisfied. The solution
 to the problem with homogeneous boundary condition is constructed
 in the form \eqref{1.1}. The solution to the  problem with non-smooth data
  is determined as a limit of solutions
 for smooth   approximated data.

The convergence of the Taylor series in the corresponding spaces is proved on the basis of the existence
result for corresponding data.

Numerical solution of a parabolic problem with large convection, when one of
the coefficients of the equation by the derivative with respect to some $x_i$ is
large for the norm of $L^\infty(Q)$, is a very difficult problem. There are many publications
dealing with these problems. Many methods s where developed for
 numerical solution of such  problems, see e.g.\ \cite{Ad.,Bu.,Do.}. However, for significantly large convection, this problem is practically
 not solved.

The method proposed in this paper permits one to construct exact solutions to such problems for infinitely
differentiable approximations of the right-hand side $f$ and initial data $u_0$.
Moreover, if an approximation of $f$ is represented in the
form of a finite sum of terms  in the Taylor expansion in $t$ with coefficients
depending on $x$, then the exact solution for this approximation of $f$ is also represented in the form of a finite sum of the
Taylor  expansion. The exact solution to the problem for given data is the limit of solutions for smooth
approximations of $f$  and $u_0$.

Thus, the suggested  method of construction of  solutions to parabolic and
hyperbolic equations
 is an alternative to  methods of numerical
solution of parabolic and hyperbolic equations.

Below in Section 2, we consider  problems  for  linear and nonlinear  parabolic equations.
 Regular solutions to these equations with homogeneous and
 nonhomogeneous boundary conditions are constructed. In the case of a
nonhomogeneous boundary condition, the solution is represented  as
a sum of a function  satisfying the boundary  condition and
a limit of solutions to the this problem with zero boundary condition
 for infinitely differentiable data.  These solutions are represented in the form
\eqref{1.1}

In much the same way, we construct regular solutions to a system of parabolic
equations in Section 3.

In Section 4, we consider an initial boundary value problem for a system of hyperbolic equations
for homogeneous and nonhomogeneous boundary conditions. Solutions to these problems
 are constructed.

 In Section 5, we formulate a nonlinear problem on vibration of an orthotropic plate in a
viscous medium. We show that there exists a unique solution to this problem,
 and this solution is obtained as a limit of solutions $u^n$  to this problem for
corresponding approximations of the data of the problem; the functions $u^n$ are
 computed  in the form of Taylor expansion.

In Section 6, we consider a 3-dimensional problem for Maxwell equations and a problem on
 diffraction of electromagnetic wave by a superconductor, i.e., a slotted antenna's problem.
 Solutions to these problems are constructed.

\section{Parabolic equations}
\subsection{Linear problem and Taylor expansion}

Let $\Omega$ be a bounded domain in ${\mathbb{R}}^n$ with a boundary $S$ of the class $C^\infty$.
Let $Q = \Omega \times (0,T)$, where $T\in (0,\infty)$. Consider the problem: Find $u$ such that
\begin{gather}
\frac{\di u}{\di t} - a_{ij}(x,t)\,\frac{\di^2 u}{\di x_i \di x_j} + a_i(x,t)\,\frac{\di u}{\di x_i} + a(x,t) u =
f\,\,\text{ in } Q, \label{2.1} \\
u|_{t=0} = u_0 \,\,\text{ in }\Omega, \quad    u(\cdot,0)|_S=  u_0|_S.  \label{2.2}
\end{gather}
Here and below the Einstein convention on summation over repeated index is applied.
As seen from \eqref{2.2}, we prescribe the value of the function $u$ on the
boundary at the point $t=0$ only.

Since the boundary $S$  is of the class $C^\infty$, we can assume that the coefficients of
 equation \eqref{2.1} and the right-hand side $f$ are given in a bounded domain
 $Q_1 = \Omega_1 \times (0,T)$, where $\Omega_1\supset \overline{\Omega}$, and $u_0$
 is prescribed in $\Omega_1$, see \cite{Lio.1}, Theorem 9.1, Chapter 1.

We denote  the space of infinitely differentiable functions  with support
in $\Omega_1$ by $\mathcal{D}(\Omega_1)$, and  the space of infinitely differentiable functions on
$\Omega_1 \times [0,T]$ with support
in $\Omega_1$ for each $t \in [0,T]$   by  $C^\infty( [0,T];\mathcal{D}(\Omega_1))$.

Topologies in both  $\mathcal{D}(\Omega_1))$    and  $C^\infty( [0,T];\mathcal{D}(\Omega_1))$
are defined by the families of  corresponding seminorms.

We assume  that
\begin{equation}\label{2.3}
(f,u_0) \in U,
\end{equation}
where
\begin{align}
&U = \Big\{(f,u_0)\mid f\in C^{\infty}( [0,T];\mathcal{D}(\Omega_1)),\ f(x,t) = \sum_{k=0}^{\infty}\,
\frac 1{k!}\, \frac{\di^k f}{\di t^k}(x,0)t^k,                                           \notag   \\
&(x,t) \in \overline{\Omega}_1\times[0,T] = \overline{Q}_1,\, u_0 \in \mathcal{D}(\Omega_1)\Big\},
                                                                                                     \label{2.3a}\\
&a_{ij} \in C^{\infty}(\overline{Q}_1), \quad a_{ij}(x,t) = \sum_{k=0}^{\infty}\,\frac 1{k!}\,
         \frac{\di^k  a_{ij}}{\di t^k}(x,0)t^k, \quad i,j = 1,\dots,n,           \notag  \\
&a_{ij}(x,t)\xi_i \xi_j \ge \mu \xi^2, \quad \mu>0, \quad (x,t) \in Q_1, \quad  \xi_i,\xi_j \in \mathbb{R},
          \quad \xi^2 = \xi_1^2 + \dots +\xi_n^2,                                       \label{2.3b}\\
&a_i \in C^{\infty}(\overline{Q}_1), \quad a_i(x,t) = \sum_{k=0}^{\infty}\,\frac 1{k!}\,
         \frac{\di^k  a_i}{\di t^k}(x,0)t^k,                                                   \notag  \\
&a \in C^{\infty}(\overline{Q}_1), \quad a(x,t) = \sum_{k=0}^{\infty}\,\frac 1{k!}\,
         \frac{\di^k  a}{\di t^k}(x,0)t^k.                                                        \label{2.3d}
\end{align}
A topology on $U$ is defined by the product of the topologies of $C^{\infty}( [0,T];\mathcal{D}(\Omega_1))$ and
$\mathcal{D}(\Omega_1)$.

Denote
\begin{equation}\label{2.5}
A\bigg(x,t,\frac{\di}{\di x}\bigg)u = a_{ij}(x,t)\frac{\di^2 u}{\di x_i\di x_j} - a_i(x,t)\frac{\di u}{\di x_i} - a(x,t)u.
\end{equation}
Then equation \eqref{2.1} can be represented in the form
\begin{equation}\label{2.6}
\frac{\di u}{\di t} - A\bigg(x,t,\frac{\di}{\di x}\bigg)\,u =f \quad \text{ in } Q.
\end{equation}

We differentiate equation \eqref{2.6} in $t$ $k-1$ times  and set $t=0$. This
gives the following recurrence relation:
\begin{gather}
\frac{\di^k u}{\di t^k} (\cdot,0) =
\bigg(\frac{\di^{k-1}}{\di t^{k-1}}\bigg(A\bigg(x,t,\frac{\di}{\di x}\bigg)u\bigg)\bigg)
                                                      (\cdot,0) + \frac{\di^{k-1} f}{\di t^{k-1}}(\cdot,0) \notag\\
                       =  \sum_{j=0}^{k-1} \,C_{k-1}^j \bigg(\frac{\di^j A}{\di t^j}\bigg(x,t,\frac{\di}{\di x}\bigg)
                    \bigg)(\cdot,0)\bigg(\frac{\di^{k-1-j}}{\di t^{k-1-j}}u\bigg)(\cdot,0)
                                                         + \frac{\di^{k-1} f}{\di t^{k-1}}(\cdot,0), \quad
                                                     k=1,2,\dots       \label{2.7}
\end{gather}
Here $u(\cdot,0) = u_0$, $C_{k-1}^j $  are the binomial coefficients, $\frac{\di^j A}{\di t^j}(x,t,\frac{\di}{\di x})$
is the operator  obtained from the operator $A$ by differentiation of its coefficients in $t$ $j$ times.

A smooth solution $u$ satisfies the condition
\begin{equation}\label{2.11}
\frac{\di^m  u}{\di t^m}(x,0) = \frac{\di^m  u_b}{\di t^m}(x,0),  \quad
x\in S.
\end{equation}
For $m=0$, we get $u_0(x) = u_b(x,0)$, $x\in S$.

We say  that the compatibility condition of order $k$ is satisfied if  \eqref{2.11}
holds for $m =0,1,2,\dots,k$.

For infinitely differentiable solutions
the compatibility condition  of order $k=\infty$  is satisfied.

\begin{theorem}\label{tyry7}
Let $\Omega$ be a bounded domain in ${\mathbb{R}}^n$ with a boundary $S$ of the class $C^{\infty}$
and $T\in (0,\infty)$. Suppose that the conditions \eqref{2.3}--\eqref{2.3d} are satisfied. Then
there exists a unique solution to the problem \eqref{2.1}, \eqref{2.2}  such that $u \in C^{\infty}(\overline{Q})$,
and this solution is represented in the form of a Taylor expansion
\begin{equation}\label{2.9}
u(x,t) = u_0(x) +\sum_{k=1}^{\infty}\,\frac 1{k!}\, \frac{\di^k  u}{\di t^k}(x,0)\,t^k, \quad
(x,t) \in \overline{Q}.
\end{equation}
The coefficients $\frac{\di^k  u}{\di t^k}(\cdot,0)$ are defined by the recurrence relation \eqref{2.7}.
Furthermore, the boundary condition function $u_b = u|_{S_T}$, $S_T = S\times [0,T]$, is
 determined as follows:
\begin{equation}\label{2.10}
u_b(x,t) = u_0(x) +\sum_{k=1}^{\infty}\,\frac 1{k!}\, \frac{\di^k  u}{\di t^k}(x,0)\,t^k, \quad
x\in S, \  t\in [0,T].
\end{equation}
The function $(f,u_0) \mapsto u$ defined by the solution to the problem \eqref{2.1}, \eqref{2.2}
is a continuous mapping of $U$ into $C^{\infty}(\overline{Q})$.
\end{theorem}

\begin{proof}

We consider the problem: Find $\check{u}$ satisfying
\begin{gather}
\frac{\di \check{u}}{\di t} - A\bigg(x,t,\frac{\di }{\di x}\bigg) \check{u} = f
                                                                            \,\,\text{ in }Q_1, \notag \\
\check{u}|_{t=0} = u_0 \,\,\text{ in }\Omega_1, \quad \check{u}|_{S_{1T}} = 0,                                                                                                         \label{2.14}
\end{gather}
where $S_{1T} = S_1 \times [0,T]$, $S_1$  is the boundary of $\Omega_1$. By  \eqref{2.3a},  $S_1$ is of the class $C^{\infty}$.
It follows from  \eqref{2.3}, \eqref{2.7} and \eqref{2.14} that the compatibility condition
of any order $k \in \mathbb{N}$  is satisfied, and by \cite{LSU.}, Theorem 5.2, Chapter IV and
 \cite{Sol.2}, Theorem 5.4, Chapter V,  there exists
 a  unique solution to the problem \eqref{2.14} such that $\check{u} \in W_q^{2(k+1),k+1}(Q_1)$,
$k \in \mathbb{N} = \{0,1,2,\dots\}$, $q\ge 2$. Therefore,
$\check{u} \in C^{\infty}({\overline{Q}}_1)$.

The functions $ \frac{\di^k \check{u} }{\di t^k}(\cdot,0)$ are defined by  formula \eqref{2.7} in which  $\Omega$
is replaced by  $\Omega_1$ and $u$ by  $\check{u}$.
 \eqref{2.3a} implies  $\frac{\di^k \check{u} }{\di t^k}(\cdot,0) \in \mathcal{D}(\Omega_1)$,
$k \in \mathbb{N}$.

Informally, the solution to the problem  \eqref{2.14} is represented in the form of a Taylor expansion
\begin{equation}\label{2.15}
\check{u}(x,t) = u_0(x) + \sum_{k=1}^{\infty} \, \frac{1}{k!}\frac{\di^k \check{u}}{\di t^k}(x,0) t^k, \quad
(x,t) \in Q_1.
\end{equation}
The function $\check{u}$ defined by \eqref{2.15}  represents a smooth solution to the problem
\eqref{2.14} for all points $t\in [0,T]$ such
that the series \eqref{2.15}  converges at $t$ in $\mathcal{D}(\Omega_1)$.

Let us prove this.
Denote
\begin{equation}\label{2.16}
\check{u}_m (x,t) = u_0(x) +\sum_{k=1}^m\,\frac 1{k!} \,
                                                       \frac{\di^k \check{u}}{\di t^k}(\cdot,0)\,t^k.
\end{equation}
\eqref{2.7} and \eqref{2.16} imply that
the function $\check{u}_m$ is a solution to the problem
\begin{align}
&\frac{\di \check{u}_m}{\di t} - A\bigg(x,t,\frac{\di }{\di x}\bigg) \check{u}_{m-1} =f_{m-1}
                                                                            \,\,\text{ in } Q_1, \notag \\
&\check{u}_m|_{t=0} = u_0 \,\,\text{ in }\Omega_1, \quad \check{u}_m|_{S_{1T}} = 0,
\label{2.17}
\end{align}
where
\begin{equation}\label{ab}
f_{m-1} (x,t) = \sum_{k=0}^{m-1}\,\frac 1{k!}\,
      \frac{\di^k f}{\di t^k}(x,0)\,t^k, \quad (x,t) \in Q_1.
\end{equation}

It follows from \eqref{2.3a}  that
\begin{equation}\label{2.3c}
f_m \to f  \,\,\text{ in }C^{\infty}([0,T];\mathcal{D}(\Omega_1)).
\end{equation}

It is known that the solution of a parabolic problem depends continuously on
the data of the problem $f, u_b, u_0$, see \cite{LSU.}, Theorem 5.2, Chapter IV and
 \cite{Sol.2}, Theorem 5.4, Chapter V. Because of this,
 \eqref{2.14} and  \eqref{2.17} yield
\begin{equation}\label{ad}
\|\check{u}-\check{u}_m\|_{W_q^{2(k+1),k+1}(Q_1)}\le c\|f_m-f\|_{W_q^{2k,k}}(Q_1), \quad k \in \mathbb{N}, \ q\ge 2.
\end{equation}
Therefore
\begin{equation}\notag
\check{u}_m \to \check{u}\,\,\text{ in } W_q^{2(l+1),l+1}(Q_1), \quad  l\in \mathbb{N},
\  q\ge 2,
\end{equation}
and $ \check{u}_m \to \check{u}$ in $C^{\infty}({\overline{Q}}_1)$.
The function $u = \check{u}|_Q$ is a solution to the problem \eqref{2.1}, \eqref{2.2},
and it is determined by \eqref{2.7} and \eqref{2.9}. This solution is unique.

It follows from \cite{LSU.}  and \cite{Sol.2} that the function $(f,u_0)\mapsto \check{u}$ defined by the solution to the
problem \eqref{2.14}, is a continuous mapping of $U$ into $ W_q^{2(l+1),l+1}(Q_1)$ for any
$l\in \mathbb{N}$, $q\ge 2$.
Therefore, the function $(f,u_0)\mapsto u$, where $u$ is the solution to the problem  \eqref{2.1},
\eqref{2.2}, is a continuous mapping of $U$ into $C^{\infty}(\overline{Q})$.
\end{proof}

\begin{remark}
 It is customary to prescribe for a parabolic equation the functions $f, u_0$ and the boundary
 condition $u_b$. However, it follows from Theorem 2.1 that under the conditions of this theorem, one
 prescribes only $f$ and $u_0$. In this case,
there exists a unique solution to the problem \eqref{2.1}, \eqref{2.2} that is represented in the form
 \eqref{2.9} and the function $u_b$ is determined by $f$ and $u_0$.
\end{remark}

\begin{corollary}
 Let $f$  be a function in $Q$ that is represented in the  form of the Taylor expansion in $t$ with
 coefficients depending on $x$.    Let $u$ be a solution to the problem \eqref{2.1} such that
 $u \in C^{\infty}(\overline{Q}).$  Then, for any fixed point
$x \in \overline{\Omega}$, the partial function $t \mapsto u(x,t)$ is analytical, and $u$ is represented in the form \eqref{2.9}.
\end{corollary}

\begin{proof}
Indeed, in this case, $f \in C^{\infty}(\overline{Q})$, $u_0 \in C^{\infty}(\overline{\Omega})$,
 the compatibility condition of order infinity is satisfied, and it follows from the  Theorem 2.1
that $u$ represented in the form \eqref{2.9}.
\end{proof}

Consider the following problem on existence of an infinitely differentiable solution to a parabolic
problem with given boundary and initial conditions:
\begin{align}
&\frac{\di u}{\di t} - A\bigg(x,t,\frac{\di}{\di x}\bigg)\,u =f \quad \text{ in } Q.
                                                                    \label{2.20}     \\
&u|_{S_T} = u_b, \quad u|_{t=0} = u_0.            \label{2.21}
\end{align}
Here $A\bigg(x,t,\frac{\di}{\di x}\bigg)$ is defined by \eqref{2.5}.

We define  the following spaces:
\begin{align}
&X = \big\{u\mid u \in C^{\infty}({\overline{Q}}_1), \,\, \supp u(\cdot,t)\subset \Omega_1,
                     \,\, t\in [0,T],   \notag  \\
&  u(x,t) =  \sum_{k=0}^{\infty}\,\frac 1{k!}\, \frac{\di^k u}{\di t^k}(x,0)\,t^k,
                   \,\, (x,t) \in {\overline{Q}}_1\big\},    \notag \\
&X_0 = \{u\mid u \in X,\,\,  u\big|_{S_T}=0,
                                         \,\, u(\cdot,0) = 0 \}, \notag  \\
&Z=\{(v,w)\mid v = h(\cdot,0)|_{\overline{\Omega}}, \,\, w = h|_{S_T}, \,\, h\in X\}.
                                                                            \label{2.22}
\end{align}
We define an operator $\gamma:X \to Z$ by
\begin{equation}\notag
\gamma(u) = (u(\cdot,0)|_ {\overline{\Omega}}\,,\,u|_{S_T}).
\end{equation}
Note that $X_0 $ is the kernel of the operator $\gamma$.

Let $X/X_0$ be the factor space. If $u^1$ and $u^2$ are elements of $X$ such that
$u^1-u^2\in X_0$, then  $u^1$ and $u^2$ belong to the same class in $X/X_0$,
 say  $\overline{u}$. We say that $\overline{u}$ is of class $(u_0,u_b)\in Z$
if $\gamma(u) = (u_0,u_b)$ for all $u\in \overline{u}$.
The result bellow follows from Theorem \ref{tyry7}.
\begin{corollary}
Let $(u_0,u_b)\in Z$ and let  $\overline{u}$ be the class $(u_0,u_b)$ from $X/X_0$.
Then any function  $u\in \overline{u}\big|_{\overline{Q}}$ is the solution to the problem \eqref{2.20}, \eqref{2.21}
for $u_0$, $u_b$, and $f$ that is determined as follows:
\begin{equation}\label{2.23}
f(x,t) = \sum_{k=1}^{\infty}\,\frac 1{(k-1)!}\, \frac{\di^{k-1} f}{\di t^{k-1}}(x,0)\,t^{k-1}, \quad
(x,t) \in \overline{Q},
\end{equation}
where
\begin{gather}
\frac{\di^{k-1} f}{\di t^{k-1}}(x,0) = \frac{\di^k u}{\di t^k}(x,0) - \sum_{j=0}^{k-1}\,C_{k-1}^j
\bigg(\frac{\di^j A}{\di t^j}\bigg(x,t,\frac {\di}{\di x}\bigg)\bigg)(x,0)   \notag\\
\times    \bigg(\frac{\di^{k-1-j} u}{\di t^{k-1-j}}\bigg) (x,0), \quad x\in \Omega. \label{2.24}
\end{gather}
\end{corollary}
Define the following set:
\begin{align}
&U_1 = \Big\{(f, u_0, u_b)\mid f\in C^{\infty}( [0,T];\mathcal{D}(\Omega)),\, f(x,t) = \sum_{k=0}^{\infty}\,
\frac 1{k!}\, \frac{\di^k f}{\di t^k}(x,0)t^k,                                           \notag   \\
&(x,t) \in \overline{\Omega}\times[0,T] = \overline{Q},\, u_0 \in \mathcal{D}(\Omega), \, u_b=0 \Big\}.
\end{align}                                                                                                     \label{2.3f}\\
We consider the problem: Given $(f, u_0, u_b) \in U_1$, find $u$ such that
\begin{gather}
u \in C^{\infty}(\overline{Q}),\notag \\
\frac{\di u}{\di t}-A\big(x,t,\frac{\di u}{\di t}\big)u=f, \notag \\
u|_{t=0}=u_0, \,\,\,u|_{S_T}=u_b. \label{2.3g}
\end{gather}
The following result follows from the proof of Theorem 2.1:
\begin{corollary}
Let $\Omega$ be a bounded domain in $\mathbb{R}^n$ with a boundary $S$ of  class $C^{\infty}$.
Suppose that the conditions \eqref{2.3b}, \eqref{2.3d} are satisfied, and $(f, u_0, u_b) \in U_1$.
Then there exists a unique solution to the problem \eqref{2.3g} that is represented in form \eqref{2.9},
 and the function $(f, u_0, 0) \mapsto u $
is a continuous mapping of $U_1$ into $C^{\infty}([0,T];\mathcal{D}(\Omega))$.
\end{corollary}

\subsection{Solution of initial-boundary value problems  in Sobolev
space}

We consider the problem \eqref{2.20},  \eqref{2.21} in which we are given $f,u_0,u_b $.
We suppose that
 \begin{equation} \label{2115}
f \in L^2(Q), \quad u_0 \in H^1(\Omega), \quad
     u_b \in  H^{\frac32, \frac34}(S_T),  \quad u_0(x)=u_b(x,0) \,\,\, x  \in S.
\end{equation}
In this  case, the compatibility condition of order zero is satisfied.

For the sake of simplicity, we assume that the coefficients of the equation \eqref{2.20}
are elements of $C^{\infty}(\overline{Q})$, and they are represented in the form of
Taylor expansion in $t$ with coefficients depending on $x$, i.e., \eqref{2.3b}, \eqref{2.3d} hold,
and the boundary $S$ is of the class $C^\infty$.

It follows from the corollary to the Stone--Weierstrass theorem that the set of
tensor products of polynomials in $x$ and polynomials in
$t$ is dense in $C^{\infty}(\overline{Q})$. Therefore, the solution to the equation
with non-smooth coefficients is obtained as the limit of solutions of equations with  smooth
 coefficients as above.

By analogy, a non-smooth boundary $S$ can be approximated by boundaries of the class $C^{\infty}$.
In this case, solutions for smooth boundaries converge to the solution for non-smooth
boundary  in the corresponding space, see
\cite{Lit.6}.

It follows from the known results, see e.g.  \cite{LiM.1}, Chapter  4, Theorems 2.3 and 6.2,
\cite{Sol.2}, Chapter V, Theorem 5.4 that, under the above conditions,   there exists
a unique solution to the problem \eqref{2.20}, \eqref{2.21} such that
\begin{equation}\label{2116}
u \in H^{2,1}(Q).
\end{equation}

We define the following function:
\begin{gather}
 w(x,t)=
 \begin{cases}
  u_b(Px,t)e^{1-\frac{a^2}{a^2-(x-Px)^2}}&\text{if } |x-Px|<a, \\
  0&\text{if } |x-Px|\ge a.
 \end{cases}
 \label{akm}
 \end{gather}
 Here $x\in \Omega, \,t \in (0,T)$,
 $P$ is the operator of projection of points of $\Omega$ onto $S$,
 $a$ is a small positive constant. Note that $w \in  H^{2,1}(Q).$

 Let
\begin{equation}\label{2118}
\tilde{u} = u -w.
\end{equation}
\eqref{2116}, \eqref{akm} imply
\begin{equation}
\tilde{u} \in  H^{2,1}(Q),\quad\tilde{u}|_{S_T}=0, \quad
\tilde{u}|_{t=0}=u_0-w|_{t=0} \in  H^1_0(\Omega). \label{ab1}
\end{equation}
The function $\tilde{u}$ is the solution to the following problem:
\begin{gather}
\frac{\di\tilde{u}}{\di{t}}-
A\bigg(x,t,\frac{\di}{\di{x}}\bigg)\tilde{u}=\tilde{f}, \notag \\
 \tilde{u}|_{S_T}=0, \quad  \tilde{u}|_{t=0}=u_0-w|_{t=0} \in H^1_0(\Omega),\label{ab2}
\end{gather}
where
\begin{equation}
\tilde{f}=f-\frac{\di{w}}{\di{t}}+A\bigg(x,t,\frac{\di}{\di{x}}\bigg)w \in  L^2(Q).\label{ab3}
\end{equation}
 \eqref{2115},  \eqref{2118}, \eqref{ab2}
imply
\begin{equation}
\tilde{u}(x,0) = u_b(x,0)-w(x,0)=0, \quad x \in S.\label{2117}
 \end{equation}
 Therefore, the compatibility condition of order zero is satisfied, and there
 exists a unique solution to the problem \eqref{ab2} such that $\tilde u \in H^{2,1}(Q)$,  $ u|_{S_T}=0$.

 Let $\{\tilde{f}_m, \tilde{u}_{0m}\}$ be a sequence such that
 \begin{align}
& \tilde{f}_m \in C^{\infty}([0,T],\mathcal{D}(\Omega)), \quad \tilde{f}_m (x,t) =
         \sum_{k=0}^{\infty}\,\frac 1{k!}\,\frac {\di^k \tilde{f}_m }{\di t^k}(x,0)t^k,
\quad \tilde{f}_m  \to \tilde{f} \text{ in } L^2(Q), \notag \\
&\tilde{u}_{0m} \in \mathcal{D}(\Omega),
\quad\tilde{u}_{0m}\to u_0 - w|_{t=0}
\text{ in }H^1_0(\Omega).
\label{2122}
 \end{align}
Consider the problem: Find $\tilde{u}_m$  such that
  \begin{align}
& \tilde{u}_m \in C^{\infty}([0,T];\mathcal{D}(\Omega)),  \notag\\
& \frac{\di \tilde{u}_m }{\di t} -  A\bigg(x,t,\frac{\di}{\di x}\bigg)\tilde{u}_m = \tilde{f}_m,  \notag\\
& \tilde{u}_m (\cdot,0) = \tilde{u}_{0m}.            \label{2123}
 \end{align}
We can assume that the functions $\tilde{f_m}$  and $ \tilde{u}_{0m}$     are extended  by zero to domains
$Q_1$  and $\Omega_1$ so that $(\tilde{f_m},\tilde{u}_{0m})  \in U$, see \eqref{2.3a}. Then by Theorem 2.1,
there exists a unique solution to our  problem in $Q_1$, and it is determined by  \eqref{2.15}, where the
functions $\check{u}$ and  $\check{u}_0$ are replaced by   $\tilde{u}_m$  and   $\tilde{u}_{0m}$. Thus, the solution
to the problem \eqref{2123}   belongs  to   $C^{\infty}(\overline{Q})$  and it is
represented in the form
\begin{equation}
\tilde{u}_m (x,t)= \tilde{u}_{0m}(x) + \sum_{k=1}^{\infty}\frac1{k!}\frac{\di^k \tilde{u}_m}{\di t^k}(x,0)t^k,
 \quad (x,t)\in Q, \label{eb}
\end{equation}
where $\frac{\di^k \tilde{u}_m}{\di t^k}$ are determined by \eqref{2.7}.

\eqref{ab2}, \eqref{2123} and
\cite{Sol.2}, Theorem 5.4, Chapter V imply
\begin{equation}\label{2125}
\|\tilde{u}_m  - \tilde{u}\|_{H^{2,1}(Q)} \le c(\|\tilde{f}_m - \tilde{f}\|_{L^2(Q)} + \|\tilde{u}_{0m} - \tilde{u}_0+w|_{t=0}\|_{H^1_0 (\Omega)}.
\end{equation}
Now, \eqref{2122} yields
\begin{equation}\label{2126}
\tilde{u}_m  \to \tilde{u} \text{ in } H^{2,1}(Q).
\end{equation}

Thus, we have proved the following result:

\begin{theorem}
Let $\Omega$ be a bounded domain in ${\mathbb{R}}^n$ with a boundary $S$ of the class
$C^\infty$, $T\in (0,\infty)$,  and let the conditions \eqref{2.3b}, \eqref{2.3d}  be satisfied.
Let also the conditions \eqref{2115}
 be satisfied.
 Then, there exists a unique solution to the problem  \eqref{2.20}, \eqref{2.21} that satisfies
\eqref{2116}, and it  is represented
in the form $u = \tilde{u} + w$, where $w$ is given by   \eqref{akm}, and $\tilde{u}$
is determined by \eqref{2126}.
\end{theorem}

\begin{remark}
 When applying Theorem 2.2  in practice, one can compute the solution to the problem
 \eqref{2123} directly by \eqref{eb}  and  \eqref{2.7}, without  the extension that was used to prove  converges of the series.
\end{remark}

\begin{remark} Theorem 2.2 also holds in the case
where $u_b=0$. Indeed, we just need to take $w=0$ in the above computations.
\end{remark}

\begin{remark}  In practical applications the data of the problem are usually not accurately given, often they are
determined by intuition, or even are plucked out of thin air. Therefore, in such a case, it makes no sense  to solve
the problem \eqref{2123}
for a series of functions $\{\tilde{f}_m, \tilde{u}_{0m}\}$ which satisfy the condition \eqref{2122}. It is sufficient to
solve problem \eqref{2123} for one or two pairs $\tilde{f}_m, \tilde{u}_{0m}$  which are close  to
$\tilde{f}$  and $\tilde{u}_{0}$     with not  a high precision. Moreover, here
${\tilde{f}_m}$ can be taken in a form of a finite sum of the Taylor expansion,
which is suitable for a given $T$. In this case, if $\tilde{f}_m=\tilde{f}_{ml}$, where
\begin{equation}
\tilde{f}_{ml} (x,t) =
         \sum_{k=0}^l\,\frac 1{k!}\,\frac {\di^k \tilde{f}_m }{\di t^k}(x,0)t^k, \label{2gd}
\end{equation}
 then exact solution to the problem \eqref{2123} is the function $\tilde{u}_m (x,t)=\tilde{u}_{m(l+1)} (x,t) $ where
\begin{equation}
\tilde{u}_{m(l+1)} (x,t)= \tilde{u}_{0m}(x) + \sum_{k=1}^{l+1}\frac1{k!}\frac{\di^k \tilde{u}_m}{\di t^k}(x,0)t^k,
 \quad (x,t)\in Q, \label{2ge}
\end{equation}
see \eqref{2.16}, \eqref{2.17}, \eqref{ab}.
\end{remark}

\begin{remark}
Numerical solution of the problem \eqref{2.20}, \eqref{2.21} in the case of a large convection, when the
norm of one of the coefficients $a_i$ of the operator $A$ is large in $L^\infty(Q)$ is a very hard problem.
Our method permits one to construct the exact solution to the problem \eqref{2123}  in the form \eqref{eb}.
Moreover, if $\tilde{f_m} = \tilde{f_{ml}}$  is represented in the form \eqref{2gd}, then the solution to the problem
\eqref{2123} is represented in the form \eqref{2ge}.
\end{remark}

\subsection{Nonlinear parabolic equation}
We consider the following problem:
\begin{align}
&\frac{\di u}{\di t} - A\bigg(x,t,\frac{\di}{\di x}\bigg)\,u + \bigg(b_0(x,t)u^2 + b_i(x,t)\frac{\di u}{\di x_i} u \notag \\
&         + b_{ij}(x,t)\, \frac{\di u}{\di x_i}\,\frac{\di u}{\di x_j}\bigg)\lambda  = f \,\,\text{ in } Q, \quad
             i,j = 1,2,\dots,n,         \label{2.58} \\
&u|_{t=0} = u_0, \,\, u|_{S_T}=0.                 \label{2.59}
\end{align}
As before, $Q = \Omega\times (0,T)$, $\Omega$ is a bounded domain in  ${\mathbb{R}}^n$ with
a boundary $S$ of the class $C^{\infty}$, $T<\infty$.

We suppose that $A\big(x,t,\frac{\di}{\di x}\big)$ is defined by \eqref{2.5}  and the conditions
\eqref{2.3b}, \eqref{2.3d}  are satisfied. Furthermore,
\begin{align}
&f\in C^{\infty}([0,T];\mathcal{D}(\Omega)),\,\, f(x,t)=  \sum_{k=0}^{\infty}\,\frac 1{k!}\,
\frac{\di^k f}{\di t^k}(x,0)\,t^k, \,\, u_0 \, \in \mathcal{D}(\Omega),  \notag \\
&b_0 \in  C^{\infty}(\overline{Q}), \,\,\,b_0(x,t) =  \sum_{k=0}^{\infty}\,\frac 1{k!}\, \frac{\di^k b_0}{\di t^k}(x,0)\,t^k,         \notag \\
&b_i \in  C^{\infty}(\overline{Q}), \,\, \,b_i (x,t) =  \sum_{k=0}^{\infty}\,\frac 1{k!}\, \frac{\di^k b_i}{\di t^k}(x,0)\,t^k, \,\,
i = 1,2,\dots,n,        \notag\\
&b_{ij} \in  C^{\infty}(\overline{Q}), \,\, \,b_{ij} (x,t) = \sum_{k=0}^{\infty}\,\frac 1{k!}\,
                \frac{\di^k b_{ij}}{\di t^k}(x,0)\,t^k, \,\,\, i,j = 1,2,\dots,n,         \label{2.60}
\end{align}
and $\lambda$ is a small positive parameter, $\lambda \in (0,\check{\lambda}]$,
$\check{\lambda} > 0$.
We define the following mapping:
\begin{equation}\label{2.61}
M(u) = b_0 u^2 + b_i\, \frac{\di u}{\di x_i} \, u + b_{ij} \frac{\di u}{\di x_i}\, \frac{\di u}{\di x_j}.
\end{equation}
Equation \eqref{2.58} can be represented  in the form
\begin{equation}\label{2.62c}
\frac{\di u}{\di t} -  A\bigg(x,t,\frac{\di}{\di x}\bigg)\,u  + \lambda M(u) = f.
\end{equation}
We differentiate equation \eqref{2.62c}  in $t$  $k-1$  times and set $t = 0$.
We obtain the relations
\begin{align}
&\frac{\di^k u}{\di t^k}(\cdot,0) = \bigg(\frac{\di^{k-1}}{\di t^{k-1}}
            \bigg(A\bigg(x,t,\frac{\di}{\di x}\bigg)\,u\bigg)\bigg)(\cdot,0)
       - \lambda\bigg(\frac{\di^{k-1}}{\di t^{k-1}}\, M(u)\bigg) (\cdot,0)
       +    \frac{\di^{k-1}}{\di t^{k-1}}\,f(\cdot,0),  \notag\\
&  k = 1,2,\dots,     \label{2.63}
\end{align}
where
\begin{align}
\frac{\di^{k-1}}{\di t^{k-1}}\, M(u) = &\sum_{l=0}^{k-1} \,C_{k-1}^l \bigg(\frac{\di^l b_0}{\di t^l}\,
         \frac{\di^{k-1-l} u^2}{\di t^{k-1-l}}
        + \frac{\di^l b_i}{\di t^l}\,\frac{\di^{k-1-l}}{\di t^{k-1-l}}\bigg(\frac{\di u}{\di x_i}\,u\bigg)  \notag\\
&                 +  \frac{\di^l b_{ij}}{\di t^l}\,\frac{\di^{k-1-l}}{\di t^{k-1-l}}\bigg(\frac{\di u}{\di x_i}\,
             \frac{\di u}{\di x_j}\bigg)\bigg).                    \label{2.64}
\end{align}
\begin{theorem}
Let $\Omega$ be a bounded domain in ${\mathbb{R}}^n$ with a boundary $S$ of the class
$C^\infty$. Suppose that the conditions \eqref{2.3b}, \eqref{2.3d},  \eqref{2.60} are satisfied.
Then for any $l,q$ such that $l\in \mathbb{N}$, $(n+2)/2q < l$, $q \ge 2$, there is
$\lambda_0 > 0$   such that,
for any $\lambda\in (0,\lambda_0)$, there exists a unique solution $u=u_\lambda$
to the problem \eqref{2.58}, \eqref{2.59} such that $u_\lambda \in W_q^{2l+2,l+1}(Q)$ and
\begin{equation}\label{2.65}
u_{\lambda}(x,t) = u_0(x) + \sum_{k=1}^\infty\,\frac 1{k!}\,
              \frac{\di^k u}{\di t^k}(x,0)\,t^k, \quad (x,t) \in \overline{Q},
\end{equation}
where $ \frac{\di^k u}{\di t^k}(x,0)$ is determined  by \eqref{2.63} and \eqref{2.64}.
Furthermore,
 $\lambda \mapsto u_\lambda$ is s continuous mapping of $(0,\lambda_0)$ into
$W_q^{2l+2,l+1}(Q)$.
\end{theorem}

\begin{proof} We consider the problem: Find $u_\lambda$ satisfying
\begin{align}
&\frac{\di u_{\lambda}}{\di t} -
A\bigg(x,t,\frac{\di}{\di x}\bigg)\,u_{\lambda}  + \lambda M(u_{\lambda}) = f
                                                                        \,\,\,\text{ in } Q,      \label{2.67} \\
&u_{\lambda}|_{t=0} = u_0 \,\,\,\text{ in } \Omega,  \quad  u_{\lambda}|_{S_T} = 0.
                                                                                                                     \label{2.68}
\end{align}

Denote
\begin{align}
&W_{q,0}^{2l+2,l+1}(Q) = \big\{ w\mid w \in W_q^{2l+2,l+1}(Q), \, w(\cdot,t) \in {\overset{\circ}
               {W}}{}_q^{2l+2- \frac 2q} (\Omega)  \text{ a.e. in }  (0,T), \,\,                       \notag\\
& \qquad l > \frac {n+2}{2q}, \,\, q\ge 2\big\},        \notag
 \end{align}
where ${\overset{\circ} {W}}{}_q^{2l+2- \frac 2q}(\Omega)$ is the closure of $\mathcal{D}(\Omega)$ in  $W_q^{2l+2- \frac 2q}(\Omega)$.

 The function $M$ maps the space $W_{q,0}^{2l+2,l+1}(Q)$ into $W_{q,0}^{2l+1,l+\frac 12}(Q)$.
 Let $u,h$ be elements of $W_{q,0}^{2l+2,l+1}(Q)$. We have
\begin{align}
&\lim_{\gamma\to 0} \,\frac {M(u+\gamma h) - M(u)}{\gamma} = 2 b_0 u h + b_i
       \bigg(\frac {\di u}{\di x_i}\,h + u\,\frac {\di h}{\di x_i}\bigg)         \notag \\
&\quad       + b_{ij}\bigg(\frac {\di u}{\di x_i}\frac {\di h}{\di x_j} + \frac {\di u}{\di x_j}\frac {\di h}{\di x_i}\bigg)
        = M'(u)\,h.        \notag
 \end{align}

It is easy to see that
\begin{equation}\notag
\|M(u + h) - M(u) - M'(u)h\|_{W_{q,0}^{2l+1,l+\frac12}(Q)} \le  c \|h\|_{W_{q,0}^{2l+2,l+1}(Q)}^2.
\end{equation}
Therefore, the operator $M$ is a Fr\'{e}chet continuously differentiable mapping of
$W_{q,0}^{2l+2,l+1}(Q)$ into $W_{q,0}^{2l+1,l+\frac12}(Q)$.

By Corollary 2.4,  for $\lambda = 0$, there exists a unique solution  to the  problem
\eqref{2.67},  \eqref{2.68}  that belongs to
$C^{\infty}([0,T];\mathcal{D}(\Omega))$, and it is determined by
\eqref{2.9}.

By applying the implicit function theorem, see e.g. \cite{Sch.}, Theorem 25, Chapter III, we obtain
that for any $l$, $q$  such that $\frac {n+2}{2q} < l$, $q \ge 2$, $l \in \mathbb{N}$, there is
$\lambda_0 > 0$  such that for any
$\lambda \in (0,\lambda_0)$, there exists a unique solution
$u_{\lambda}$ to the  problem \eqref{2.67}, \eqref{2.68}
such that $u_{\lambda} \in W_{q,0}^{2l+2,l+1}(Q)$, and the function
$\lambda \mapsto u_{\lambda} $ is a continuous mapping of
$(0,\lambda_0)$ into $W_{q,0}^{2l+2,l+1}(Q)$.

Informally, the solution to the problem \eqref{2.58}, \eqref{2.59} is represented
in the form
\eqref{2.65}.
Define $u_m$ as follows:
\begin{equation}\label{2.51b}
u_m(x,t) = u_0(x) + \sum_{k=1}^m\,\frac 1{k!}\,
          \frac{\di^k u}{\di t^k}(x,0)\,t^k, \quad (x,t) \in \overline{Q}.
\end{equation}
\eqref{2.63} and \eqref{2.51b} imply that

\begin{equation}\label{2.62}
\frac{\di u_m}{\di t}(x,t) -  A\bigg(x,t,\frac{\di}{\di x}\bigg)\,u_{m-1}(x,t)  +\lambda \sum_{k=0}^{m-1}\,\frac 1{k!}
           \, \bigg(  \frac{\di^k }{\di t^k}\,(M(u))\bigg)(x,0)\,t^k  = f_{m-1}(x,t),
\end{equation}
where
\begin{equation}\label{2.53a}
f_{m-1}(x,t)=\sum_{k=0}^{m-1}\,\frac 1{k!}\frac{\di^k f}{\di t^k}(x,0)t^k.
\end{equation}
It follows from \eqref{2.60} that
\begin{equation}\label{2.54b}
f_{m-1} \to f \,\, \,\text{in} \,\, \,C^{\infty} ([0,T];\,\mathcal{D}(\Omega)).
\end{equation}
By \eqref{2.60}  and \eqref{2.63}, we have
 $u_m \in C^{\infty} ([0,T];\,\mathcal{D}
 (\Omega))$.

We apply  the implicit function theorem to the case where $\lambda$ is from a small vicinity of  zero in the set of nonnegative numbers, and the right-hand side of
\eqref{2.58} belongs to a small vicinity  of $f$ in  $W^{2l,l}_{q,0}(Q)$. Then  \eqref{2.54b}
yields
\begin{equation}\label{2.54c}
u_m \to u_{\lambda} \,\, \,\text{in } W^{2l+2,l+1}_{q,0}(Q).
\end{equation}
Therefore, the series \eqref{2.65} converges in
$ W^{2l+2,l+1}_{q,0}(Q)$,   and gives the solution
to the problem \eqref{2.58},  \eqref{2.59}.
\end{proof}

We remark that, in the case $f \in L^2 (Q)$ and $u_0 \in H^1_0(\Omega)$,
 the solution to the problem \eqref{2.58}, \eqref{2.59} can be defined as the limit of
  solutions to this problem for  $f=\tilde{f}_m$ and $u_0=\tilde{u}_{0m}$ that are determined by  \eqref{2122} with  $w=0$.

 \subsection{Construction of functions of  $\mathcal{D}(\Omega)$ }

Let $\Omega_2$  be a domain   in ${\mathbb{R}}^n$  such that $\overline{\Omega}_2  \subset \Omega,  $  and $S_2$
be the boundary of $\Omega_2$. We suppose that
\begin{equation}
 d(x,S)=2a \,\, \text{for any}\,\, x \in S_2,  \label{7.1}
 \end{equation}
 where
 \begin{equation}
 d(x,S)=\min
 \bigg(\sum_{i=1}^n
\big (x_i-y_i\big)^2\bigg)^{\frac{1}{2}}, \,\, y=(y_1, \dots, y_n) \in S,
 \end{equation}
 and $a$ is a small positive constant.

 Define the following function:
 \begin{gather}
 g_a(x)=
 \begin{cases}
  1,& \text{if } d(x,S_2)>0, \,\, x \in \Omega_2, \\
  e^{1-\frac{a^2}{a^2-(d(x,S_2))^2}},& \text{if }
  d(x,S_2) \in [0,a),
   \,\, x \in \Omega \setminus \Omega_2, \\
  0,&  \text{if } d(x,S_2) \ge a, \,\, x \in \Omega \setminus \Omega_2.
 \end{cases}
 \label{7.3}
 \end{gather}

The function $g_a$ belongs to $\mathcal{D}(\Omega)$, and if
$f \in  C^{\infty}(\overline{\Omega})$, then $w=f \cdot g_a \in \mathcal{D}(\Omega),$ and
 \begin{equation}
\text{the set} \{P_m\cdot g_a\},
  \,\, m \in\mathbb{N}, \,\, a>0  \,\,  \text{is dense in} \,\,H^l_0(\Omega),  \,\, l \in \mathbb{N}, \label{7.4}
\end{equation}
where $P_m$  is any polynomial in $x$ such that the order of polynomial in $x_i, i=1,\cdots, n$, does not exceed $m$.

In the general case, a smooth boundary $S$ is defined by local cards, i.e., by local coordinate systems
$(y^k_1, \cdots, y^k_n )$  and mappings $F_k$, $k=1, \cdots, \beta$, such that
\begin{equation}
y^k_n=F_k(y^k_1, \dots, y^k_{n-1}),         \label{7.4a}
\end{equation}
and by  a corresponding partition of unity, see e.g. \cite{Lio.1, Sch., Sol.2}.

For a ball or a paraboloid, the boundary $S$ is defined by
\begin{equation}
\omega_b \,(x)=
\sum_{i=1}^n\, x_i^2 -c^2=0, \quad \omega_p=\sum_{i=1}^n\,\frac{ x_i^2}{b_i^2} -c^2=0, \label{7.5}
\end{equation}
where $b_i$ and $c$ are positive constants.

Polyhedral domains are widely used in practical computations. For convex polyhedron, whose  faces
are defined by equations
\begin{equation}
f_k(x)=\sum_{i=1}^n \, a_{ik}x_i-c_k=0,\quad k=1,\dots, m, \label{7.6}
\end{equation}
where $a_{ik} $ and $c_k$  are  constants, the boundary $S$ is given as follows:
\begin{equation}
\omega_{cp}(x)=\pm \prod_{k=1}^m   f_k(x)=\pm \prod_{k=1}^m  \sum_{i=1}^n \, a_{ik}x_i-c_k=0, \label{7.7}
\end{equation}
where the sign is chosen so that $\omega_{cp}(x)>0$  in $\Omega.$

The domain $\Omega_2$ of polyhedron is the polyhedron with the boundary $S_2$ that satisfies
the condition \eqref{7.1}. In this case, \eqref{7.4} holds.

The boundary  of polyhedron is infinitely differentiable everywhere with exception of angular
 points, at which it is not differentiable.
 Nevertheless, in small vicinities of angular points  this boundary can be regularized by convolution
 of the function $F_k$, see   \eqref{7.4a}, with an infinitely differentiable   function with a small support, in particular, with the bump function.

If the boundary of a polyhedron is not regularized, then the computation of the solution to the problem
 in exteriors of any small vicinities of angular points can be fulfilled.

 For the case of non-convex polyhedron,  one can identify the faces of the polyhedron with local cards, without
 using a partition of unity. That is,  one assumes that $f_k$ are the identity mappings of the
 sets
 $$ G_k=\{x\mid f_k(x)=0, \,\,    f_k(x) \in S \} $$
 onto itself, and $G_k$ are defined so that
  $ \bigcup_{k=1}^m G_k=S$.

 \section{System of parabolic equations}
Let us consider the following problem for a system of equations that are parabolic  in the sense of Petrowski:
Find $u = (u_1,u_2,\dots,u_N)$ such that
\begin{align}
&\frac{\di u_i}{\di t} - \frac{\di }{\di x_r}\bigg(a_{ijrm}(x,t)\,\frac{\di u_j}{\di x_m}\bigg) +
                 b_{jm}^i\,(x,t)\,\frac{\di u_j}{\di x_m} + g_j^i (x,t)u_j = f_i \quad  \text{in } Q, \notag \\
&i,j = 1,\dots,N, \quad r,m = 1,\dots,n, \label{2.40} \\
&u|_{t=0} =u_0 \,\,\text{ in }\Omega, \,\,\, u|_{S_T }=u_b.                             \label{2.41}
\end{align}
As before,  $Q = \Omega\times (0,T)$, $\Omega$ is a bounded domain in ${\mathbb{R}}^n$
with a boundary $S$ of the class $C^\infty$, $T\in (0,\infty)$.

We suppose that
\begin{gather}
f = (f_1,\dots,f_N) \in L^2 (Q)^N, \,\,\, u_0=(u_{01}, \dots, u_{0N}) \in  H^{1}(\Omega)^N,
 \notag \\  u_b=(u_{b1},  \dots, u_{bN} ) \in   H^{ \frac{3}{ 2},\frac{3}{4}}(S_T)^N, \,\, u_0(x)=u_b(x,0), \,\, x\in S,
                                                  \label{2.46a}
\end{gather}

and
\begin{align}
& a_{ijrm} \in  C^\infty({\overline{Q}}_1),   \,\,  a_{ijrm}(x,t) =  \sum_{k=0}^\infty\,\frac 1{k!}\,
                \frac{\di^k  a_{ijrm}}{\di t^k}(x,0)\,t^k, \,\, (x,t) \in \overline{Q}_1, \notag \\
& a_{ijrm}(x,t) \xi_r\xi_m \nu_j \nu_i\ge \mu  \sum_{r=1}^n \,  \xi_r^2   \sum_{i=1}^N  \nu_i^2,  \,\,
                (x,t)\in  \overline{Q}_1, \,\,
                  \xi_r \in \mathbb{R}, \,\, \nu_i  \in \mathbb{R}, \,\, \mu>0,   \label{2.47a} \\
&b_{jm}^i \in C^\infty({\overline{Q}}_1), \,\, b_{jm}^i (x,t) = \sum_{k=0}^\infty\,\frac 1{k!}\,
                \frac{\di^k b_{jm}^i }{\di t^k}(x,0)\,t^k, \,\,\, (x,t) \in \overline{Q}_1,  \label{2.48a} \\
&g_j^i \in   C^\infty({\overline{Q}}_1), \,\, g_j^i (x,t) = \sum_{k=0}^\infty\,\frac 1{k!}\,
                \frac{\di^k g_j^i }{\di t^k}(x,0)\,t^k, \,\,\, (x,t) \in \overline{Q}_1, \label{2.49a}
\end{align}
\eqref{2.46a} yields that, the compatibility condition of order zero is satisfied. It follows from  \cite{Sol.2} that there exists
a unique solution to the problem \eqref{2.40}, \eqref{2.41} such that $u \in H^{2,1}(Q)^N$.

Informal differentiation of  \eqref{2.40} in $t$ gives the following relations:
\begin{align}
\frac{\di^k u_i }{\di t^k}(\cdot,0)& = \bigg(\frac{\di^{k-1}}{\di t^{k-1}}\bigg(B_i\bigg(x,t,\frac{\di }{\di x}\bigg)
                       u\bigg)\bigg)(\cdot,0) + \frac{\di^{k-1} f_i}{\di t^{k-1}}(\cdot,0)  \notag\\
                                                   & = \sum_{j=0}^{k-1}\, C_{k-1}^j \bigg(\frac{\di^j B_i}{\di t^j}
                                                          \bigg(x,t,\frac{\di }{\di x}\bigg)\bigg)(\cdot,0) \bigg(\frac{\di^{k-1-j}u}
                                                          {\di t^{k-1-j}} \bigg)(\cdot,0)  \notag\\
                                                    & +  \frac{\di^{k-1} f_i}{\di t^{k-1}}(\cdot,0), \quad k = 1,2,\dots,
\label{2.51a}
\end{align}
where
\begin{align}
&B\bigg(x,t,\frac{\di }{\di x}\bigg)u = \Big{\{}B_i\bigg(x,t,\frac{\di }{\di x}\bigg)u\Big{\}}_{i=1}^N,     \notag\\
&B_i\bigg(x,t,\frac{\di }{\di x}\bigg)u = \frac{\di }{\di x_r}\bigg(a_{ijrm}(x,t)\, \frac{\di u_j}{\di x_m}\bigg)
              - b_{jm}^i(x,t)\, \frac{\di u_j}{\di x_m} - g_j^i\,u_j, \quad i = 1,\dots,N. \label{2.47}
\end{align}

We mention that the inequality for $a_{ijrm}$  in \eqref{2.47a}  is the  condition of strong ellipticity of the
 operator $B$ .

Equations \eqref{2.40} are represented in the form
\begin{equation}\label{2.48}
\frac{\di u_i}{\di t} - B_i\bigg(x,t,\frac{\di }{\di x}\bigg)u = f_i\,\,\text{ in }\,\,\,Q, \,\, i = 1,\dots,N.
\end{equation}
The existence of a unique solution to the problem   \eqref{2.40},  \eqref{2.41} such that $ u \in H^{2,1} (Q)^N $
follows from \cite{Sol.2}.

By analogy with \eqref{akm}, we define the following vector-function $w=(w_1, \dots, w_N)$:
\begin{gather}
 w_i (x,t)=
 \begin{cases}
 u_{bi}(Px,t)\exp\big(1-\frac{a^2}{a^2-(x-Px)^2}\big), &\text{if } |x-Px|<a,   \ i=1,\dots, N, \\
  0,&\text{if } |x-Px|\ge a.
 \end{cases}
 \label{al}
 \end{gather}
Then $w \in H^{2,1}(Q)^N$.

Let
\begin{equation}
\tilde{u}=u-w. \label{f}
\end{equation}
The function $\tilde{u}$ is the solution to the problem
\begin{gather}
\tilde{u} \in H^{2,1}(Q)^N, \notag \\
\frac{\di{\tilde{u}}_i}{\di{t}}-
B_i\bigg(x,t,\frac{\di}{\di{x}}\bigg)\tilde{u}=\tilde{f}_i,   \,\, \text{in} \,\, Q,       \notag \\
 \tilde{u}|_{S_T}=0, \quad  \tilde{u}|_{t=0}=u_0-w|_{t=0} \in H^1_0(Q)^N, \label{f1}
\end{gather}
where
\begin{equation}
\tilde{f}_i=f_i-\frac{\di{w}_i}{\di{t}}+B_i\bigg(x,t,\frac{\di}{\di{x}}\bigg)w \in L^2(Q^N).\label{f2}
\end{equation}
Let $(\tilde{f}_m,\tilde{u}_{0m})$  be a sequence such that
\begin{align}
& \tilde{f}_m \in C^{\infty}([0,T];\mathcal{D}(\Omega))^N, \quad \tilde{f}_m (x,t) =
         \sum_{k=0}^{\infty}\,\frac 1{k!}\,\frac {\di^k \tilde{f}_m }{\di t^k}(x,0)t^k,
\quad\tilde{f}_m  \to \tilde{f} \text{ in } L^2(Q)^N, \notag \\
&\tilde{u}_{0m} \in \mathcal{D}(\Omega)^N, \,\, \,
\,\,\,\tilde{u}_{0m}\to u_0 - w|_{t=0} \,\,\,
\text{in} \,\,\, H^1_0(\Omega)^N.
\label{f3}
 \end{align}
Consider the problem: Find $\tilde{u}_m$  such that
  \begin{gather}
 \tilde{u}_m \in C^{\infty}([0,T];\mathcal{D}(\Omega))^N,  \notag\\
\frac{\di {\tilde{u}}_{mi} }{\di t} -B_i
(x,t,\frac{\di}{\di{x}}){\tilde{u}}_m=\tilde{f}_{mi}   \text{ in } Q,       \notag \\
 \tilde{u}_m(\cdot,0)= \tilde{u}_{0m}.   \label{mf}
\end{gather}
It follows from \cite{Sol.2} that there exists a unique solutions to the problem \eqref{mf}. By analogy with the above,
 we get that
\begin{equation}
  \tilde{u}_m (x,t)= \tilde{u}_{0m}(x) + \sum_{k=1}^{k=\infty}\frac1{k!}\frac{\di^k \tilde{u}_m}{\di t^k}(x,0)t^k,
\quad (x,t)\in Q, \label{ml}
\end{equation}
where $\frac{\di^k \tilde{u}_m}{\di t^k}$ are determined by \eqref{2.51a} with  $u$  and  $f$ being replaced by  $\tilde{u}_m$
and $\tilde{f}_m$, respectively.

Since the solution to the problem  \eqref{mf} depends continuously on $ \tilde{f}_m,
\tilde{u}_{0m}$
formulas
\eqref{f1}, \eqref{f3}, and \eqref{mf} imply
\begin{equation}
\tilde{u}_m \to \tilde{u}  \text{ in }  H^{2,1} (Q)^N.
\end{equation}
Thus, we have proved

\begin{theorem}
Let $\Omega$ be a bounded domain in ${\mathbb{R}}^n$ with a boundary $S$ of the class
$C^{\infty}$ and $T\in (0,\infty)$. Suppose that the conditions \eqref{2.46a}--\eqref{2.49a}
           are satisfied.
Then there exists a unique solution to the problem \eqref{2.40}, \eqref{2.41}
 such that $u \in H^{2.1}(Q)^N   $,  and this solution is represented in the form
$u=\tilde{u}+w $, where $ \tilde{u}=\lim \tilde{u}_m $ and $w$ is given by \eqref{al}.
\end{theorem}

\section{System of hyperbolic equations}
\subsection{Problem with boundary condition at $t=0$}

We consider the problem: Find $u = (u_1,u_2,\dots,u_N )$ such that
\begin{align}
&\frac {\di^2 u_i}{\di t^2} - B_i\bigg(x,t, \frac {\di}{\di x}\bigg)u = f_i \,\,\,\text{ in }\,\, Q,
\quad i = 1,2,\dots,N,                       \label{3.80}\\
&u(x,0) = u_0(x), \,\, \frac {\di u}{\di t} (x,0) = u_1(x), \,\, x \in \overline{\Omega}, \,\, u_0(x)=u_b(x,0), \,\,x \in S.
                                                          \label{3.81}
\end{align}
Here $B_i\big(x,t, \frac {\di}{\di x}\big)$ are the components of the operator $B\big(x,t, \frac {\di}{\di x}\big)$
that are defined in \eqref{2.47}.

We assume that the coefficients of the operator $B\big(x,t, \frac {\di}{\di x}\big)$ satisfy the conditions
\eqref{2.47a}--\eqref{2.49a} and
\begin{align}
&(f,u_0,u_1) \in U_2,                        \notag\\
&U_2 = \big\{(f,u_0,u_1)\mid f = (f_1,\dots,f_N) \in C^{\infty}([0,T];\mathcal{D}(\Omega_1))^N, \notag\\
& f(x,t)  = \sum_{k=0}^\infty\,\frac 1{k!}\, \frac{\di^k f}{\di t^k}(x,0)\,t^k, \,\, (x,t) \in Q_1, \notag\\
&u_0 = (u_{01},\dots,u_{0N}) \in \mathcal{D}(\Omega_1)^N,\,\,\,
                                u_1 = (u_{11},\dots,u_{1N}) \in \mathcal{D}(\Omega_1)^N \big\}.  \label{3.82}
\end{align}

We differentiate equations \eqref{3.80} in $t$ $k-2$ times, $k\ge3$,  and set $t=0$.
This gives the following recurrence relation:
\begin{align}
\frac{\di^k u_i}{\di t^k}(\cdot,0)& = \frac{\di^{k-2} f_i}{\di t^{k-2}}(\cdot,0) +
                    \bigg(\frac{\di^{k-2}}{\di t^{k-2}}\bigg(B_i\bigg(x,t,\frac{\di}{\di x}\bigg) u\bigg)\bigg)(\cdot,0)
                           \notag\\
& = \frac{\di^{k-2} f_i}{\di t^{k-2}}(\cdot,0) + \sum_{j=0}^{k-2}\, C_{k-2}^j \bigg(\frac{\di^j B_i}{\di t^j}
        \bigg(x,t,\frac{\di}{\di x}\bigg)\bigg)(\cdot,0) \frac{\di^{k-2-j} u}{\di t^{k-2-j}}(\cdot,0).
                                \label{3.83}
\end{align}
Here $u(\cdot,0)$ and $\frac{\di u}{\di t}(\cdot,0)$ are prescribed.

\begin{theorem}
Let $\Omega$ be a bounded domain in ${\mathbb{R}}^n$ with a boundary $S$ of the class
$C^\infty$ and $T\in (0,\infty)$. Suppose that the conditions \eqref{3.82} \eqref{2.47a}-\eqref{2.49a}  are satisfied.
Then there exists a unique solution to the problem \eqref{3.80}, \eqref{3.81} such that $u \in  C^{\infty}(\overline{Q})^N$,
and this solution is represented in the form of the Taylor expansion
\begin{equation}\label{3.84}
u(x,t) = u_0(x) + u_1(x)t +\sum_{k=2}^\infty\,\frac 1{k!}\, \frac{\di^k u}{\di t^k}(x,0)\,t^k, \quad
             (x,t) \in \overline{Q}.
\end{equation}
The coefficients $\frac{\di^k u}{\di t^k}(\cdot,0)$ are determined by the recurrence relations
\eqref{3.83}. Furthermore,  the boundary condition function $u_b = u|_{S_T}$ is determined as follows:
\begin{equation}\label{3.85}
u_b(x,t) = u_0(x) + u_1(x)t +\sum_{k=2}^\infty\,\frac 1{k!}\, \frac{\di^k u}{\di t^k}(x,0)\,t^k, \,\,
             (x,t) \in S_T.
\end{equation}
The function $(f,u_0,u_1)\mapsto u$ that is defined by the solution  to the problem
\eqref{3.80},  \eqref{3.81} in the form \eqref{3.84}  is a continuous mapping of $U_2$ into
$C^{\infty}(\overline{Q})^N $.
\end{theorem}

\begin{proof}
We consider the problem: Find $\check u$ satisfying
\begin{align}
&\frac {\di^2\check {u}_i}{\di t^2} - B_i\bigg(x,t, \frac {\di}{\di x}\bigg)\check u = f_i \,\,\,\text{ in }\,\, Q_1,
\quad i = 1,2,\dots,N,                       \label{4.7a}\\
&\check{u}(x,0) = u_0(x), \quad \frac {\di \check{u}}{\di t} (x,0) = u_1(x), \quad x \in \overline{\Omega}_1,
\,\, \check {u}|_{S_{1T}}=0,
\label{4.7b}
\end{align}
where $(f, u_0, u_1) \in U_2$.

It follows from \cite{LiM.1},  Chapter 5, Theorem 2.1, that under the conditions
\begin{equation}\label{4.9m}
f \in H^{0,1}(Q_1)^N, \,\,  u_0     \in H^2(\Omega_1)^N\cap H^1_0(\Omega_1)^N, \,\, u_1 \in H^1(\Omega_1)^N,
\end{equation}
there exists a unique solution to the problem \eqref{4.7a}, \eqref{4.7b} such that
\begin{equation}\label{4.10m}
\check{u} \in L^2( 0, T; H^2(\Omega_1))^N, \,\, \frac{\di^2\check{ u} }{\di t^2} \in L^2(Q_1)^N,
\end{equation}
i.e. $\check{u} \in H^{2,2}(Q_1)^N $, and the function $ (f,\, u_0, \, u_1) \mapsto \check {u} $ is a continuous
mapping of  $ H^{0,1}(Q_1)^N  \times H^2(\Omega_1 )^N  \cap H^1_0(\Omega_1)^N  \times H^1 (\Omega_1 )^N $  into $H^{2,2}(Q_1)^N  $.

Informally, the solution to the problem  \eqref{4.7a}, \eqref{4.7b}  is represented in the form
\begin{equation}\label{4.8a}
\check{u}(x,t) = u_0(x) + u_1(x)t +\sum_{k=2}^\infty\,\frac 1{k!}\, \frac{\di^k \check{u}}{\di t^k}(x,0)\,t^k, \,\,
             (x,t) \in \overline{Q}_1.
\end{equation}
The function $\check u$  defined by \eqref{4.8a} and the formula \eqref{3.83} with $u$ replaced by  $\check {u}$
 is a solution to the  problem \eqref{4.7a}, \eqref{4.7b}
for all $t \in [0,T]$ such that the series \eqref{4.8a}  converges at $t$ in the corresponding space.

Taking that into account, we conclude by analogy with the above that the series  \eqref{4.8a} converges in $H^{2,2}(Q_1)^N  $.

Consider the problem: Find a function $\hat u=(\hat {u}_1,  \dots, \hat {u}_N ) $  given in $Q_1$  that solves the problem
\begin{align}
&\frac{\di^2\hat{u}}{\di t^2} - B\bigg(x,t, \frac {\di}{\di x}\bigg)\hat{u} =\frac{\di^2 f}{\di t^2} \,\,\,\text{ in } Q_1,
                                                         \label{4.11p}\\
&\hat{u}(x,0) = \frac{\di^2 \check u}{\di t^2}(x,0), \quad \frac {\di \hat{u}}{\di t} (x,0) = \frac{\di^3 \check u }{\di t^3}(x,0),
\quad x \in \Omega_1, \notag\\
& \hat{u}(x,t) = 0, \quad      (x,t) \in S_{1T},    \label{4.12p}
\end{align}
where $\frac{\di^2 \check u}{\di t^2}(x,0)$ and $\frac{\di^3 \check u }{\di t^3}(x,0) $ are determined by \eqref{3.83}.

Again, \eqref{3.82} and \cite{LiM.1}
imply that there exists  a unique solution to the problem  \eqref{4.11p},  \eqref{4.12p} such that
  $\hat{u} \in H^{2,2}(Q_1)^N.$
As $\hat{u} =\frac{\di^2 \check u}{\di t^2} $, by \eqref{4.8a} it is represented in the form
\begin{equation}
\hat{u}(x,t)=\sum_{k=2}^\infty\,\frac 1{(k-2)!}\, \frac{\di^k \check u}{\di t^k}(x,0)\,t^{k-2}, \quad
             (x,t) \in Q_1,
\end{equation}
and
\begin{equation}\label{4.13m}
\frac{\di^4 \check u}{\di t^4} \in L^2(Q_1)^N, \,\,\quad\,\,
\frac{\di^4 \check u}{\di t^2 \di x_i^2} \in L_2(Q_1)^N , \,\, i=1, \dots, N.
\end{equation}

Now consider  the problems: Find functions $\tilde {u}_j= (\tilde {u}_{j1},\dots, \tilde {u}_{jN} )$  given in $Q_1$ such that
\begin{align}
&\frac{\di^2\tilde{u}_{ji}}{\di t^2} - B_i\bigg(x,t, \frac {\di}{\di x}\bigg)\tilde{u}_j =\frac{\di^2 f_i}{\di x_j^2} \,\,\,\text{ in } Q_1,
\,\, j=1, \dots, n, \,\, i=1, \dots, N,
                                                         \label{4.11m}\\
&\tilde{u}_j(x,0) = \frac{\di^2 \check{u}}{\di x_j^2}(x,0)=\frac{\di^2 u_0}{\di x_j^2}(x), \quad \frac {\di \tilde{u}_j}{\di  t} (x,0) = \frac{\di^3 \check{u}}{\di t \di x_j^2}(x,0)=\frac{\di^2 u_1}{\di x_j^2}(x), \notag\\
&j=1, \dots, n,
\quad x \in \Omega_1, \quad
           \tilde{u}_j(x,t) = 0, \quad      (x,t) \in S_{1T}    \label{4.12m}.
\end{align}
The preceding arguments show the existence of a unique solution to this problem such that $ \tilde{u}_j \in H^{2,2}
(Q_1)^N $.
Since $ \tilde{u}_j =\frac{\di^2  \check u}{\di x_j^2}, \,\, j=1, \dots, n $, we obtain
$$\frac{\di^4 \check u}{\di x_j^4 } \in L^2(Q_1)^N, \quad j=1, \dots, n. $$
From here and \eqref{4.13m}, we get  $\check{u} \in H^{4,4}(Q_1)^N.$

By analogy, we obtain  that $\check{u} \in H^{2k,2k}(Q_1)^N$ for any $k \in \mathbb{N}$, and
 $\check{u} \in C^{\infty}(\overline{Q}_1)^N$, and the series \eqref{4.8a} converges to
 $\check{u}$  in $C^{\infty}(\overline{Q}_1)^N.$  The function $ (f, u_0, u_1) \mapsto \check{u}$ is a continuous
 mapping of $U_2$  into $C^{\infty}(\overline{Q}_1)^N$, and
 $u=\check{u}|_Q $.
 \end{proof}

\subsection{Problem with given boundary conditions}

We first consider  the following problem with homogeneous boundary conditions:
\begin{align}
&\frac{\di^2 u_i}{\di t^2} - B_i\bigg(x,t, \frac {\di}{\di x}\bigg)\,u = f_i
                       \,\,\,\text{ in }\,\, Q,\quad i = 1,2,\dots,N,   \notag \\
&u|_{t=0}= u_0, \quad {\frac {\di u}{\di t}}\Big|_{t=0} = u_1,  \notag\\
& u|_{S_T} = 0.  \label{3.177}
\end{align}

We suppose
\begin{equation}\label{3.188}
 f \in L^2(Q)^N, \  u_0 \in H_0^1(\Omega)^N,  \  u_1 \in L^2(\Omega)^N.
\end{equation}

\begin{theorem}
Let $\Omega$ be a bounded domain in ${\mathbb{R}}^n$ with a boundary $S$ of the class
$C^{\infty}$ and $T\in (0,\infty)$. Suppose that the conditions \eqref{2.47a}--\eqref{2.49a}
and \eqref{3.188}  are satisfied.
Then there exists a unique solution to the problem \eqref{3.177} and furthermore
\begin{gather}
(f,u_0,u_1) \to \bigg(u, \frac {\di u}{\di t}\bigg)  \text{ is a linear continuous mapping of  } \notag \\
(L^2(Q)^N\times H^1_0(\Omega)^N\times L^2(Q)^N )    \text{ into } L^2(0,T; H^1_0(\Omega)^N \times L^2(Q)^N   .
\label{4.af}
\end{gather}
Let $\{f^m,u_0^m,u_1^m\}_{m=1}^{\infty}$ be a sequence that satisfies the following conditions:
\begin{gather}
 f^m \in C^{\infty}([0,T];\mathcal{D}(\Omega))^N, \,\, f^m(x,t) = \sum_{k=0}^\infty\,\frac 1{k!}\,
     \frac{\di^k f^m}{\di t^k}\, (x,0) t^k,  \quad (x,t) \in Q, \notag \\
f^m\to f  \,\,\,\text{ in } L^2(Q)^N,
\quad u_0^m\in \mathcal{D} (\Omega)^N, \,\,
 u_0^m \to u_0 \,\,\,\text{ in }H^1_0(\Omega)^{N},           \,\,        \notag \\
 u_1^m\in \mathcal{D}(\Omega)^N, \quad u_1^m \to u_1\,\,\text{ in } L^2(\Omega)^N.  \label{3.199}
\end{gather}
Let also $u^m$ be the solution to the problem
\begin{align}
&\frac{\di^2 u_i^m}{\di t^2} - B_i\bigg(x,t,\frac {\di}{\di x}\bigg)\,u^m = f_i^m,  \,\,i=1,2, \dots, N,\notag\\
&u^m(x,0) = u_0^m(x) , \quad \frac {\di u^m}{\di t}(x,0) = u_1^{m}(x), \quad  x\in \Omega.   \label{3.200}
\end{align}
Then $u^m \in   C^{\infty}([0,T];\mathcal{D}(\Omega))^N  $ and $u^m \to u \,\,\text{ in }\,\, L^2([0,T];H^1_0(\Omega))^N$,
 \, $\frac {\di u^m}{\di t}   \to   \frac {\di u}{\di t}  \,\,\text{in} \,\,L^2(Q)^N ,$
  where $u$ is the solution to the problem \eqref{3.177}.
\end{theorem}

\begin{proof} The existence of a unique solution to the problem \eqref{3.177} such that $u \in L^2([0,T];H^1_0(\Omega))^N, $
$ \frac {\di u}{\di t} \in L^2(Q)^N$  and \eqref{4.af} holds     follows from \cite{Liopt}, Chapter 4, Theorem 1.1.
 Informally, the solution to the problem
\eqref{3.200} is represented in the form
\begin{equation}\label{3.211}
u^m(x,t) = u_0^m(x) + u_1^m(x)t + \sum_{k=2}^\infty\,\frac 1{k!}\,
                         \frac{\di^k  u^m}{\di t^k}\, (x,0) t^k,  \quad (x,t) \in Q.
\end{equation}
Let $(\hat{f}^m, \hat{u}_0^m, \hat{u}_1^m)$ be an extension  of $(f^m, u_0^m, u_1^m)$ to $Q_1$ and  $\Omega_1$,
respectively, such that $(\hat{f}^m, \hat{u}_0^m, \hat{u}_1^m) \in U_2$. Then, by using Theorem 4.1, we obtain that
$u^m \in C^{\infty}(\overline{Q})^N$ and the series \eqref{3.211} converges in $C^{\infty}(\overline{Q})^N$.

Since the solution to the problem
\eqref{3.177} depends continuously on the data  of the problem, we obtain from
\eqref{3.199} that $u^m \to u$ in  $L^2([0,T];H^1_0(\Omega))^N$ and $\frac {\di u^m}{\di t}   \to
\frac {\di u}{\di t}$ in $L^2(Q)^N$.
\end{proof}

\medskip
Consider now the problem with inhomogeneous boundary conditions: Find $u$ satisfying
\begin{align}
&\frac{\di^2 u_i}{\di t^2} - B_i\bigg(x,t,\frac {\di}{\di x}\bigg)\,u = f_i
                   \,\,\text{ in }\,\,Q, \quad i = 1,2,\dots,N,  \notag\\
&u|_{t=0} = u_0,  \quad \frac{\di u}{\di t}\Big|_{t=0} = u_1,  \quad   u|_{S_T} = u_b.
 \label{3.222}
\end{align}
We suppose that
\begin{gather}
f \in L^2(Q)^N,  \quad u_b \in H^{\frac 32, \frac 32}(S_T)^N, \quad u_0 \in H^1(\Omega)^N,
                                      \notag\\
   u_1 \in L^2(\Omega)^N, \,\,\,u_0(x)=u_b(x,0),\quad x \in S.
\label{4f}
\end{gather}

We use the function $w$ defined in  \eqref{al}. Since  $u_b \in H^{\frac 32, \frac 32}(S_T)^N$, we have
$w \in H^{2,2}(\Omega)^N$.
We set
\begin{equation}  \label{4a}
\hat{u}= u-w.
\end{equation}
Then
\begin{align}
&\frac{\di^2 \hat{u}_i}{\di t^2} - B_i\bigg(x,t,\frac {\di}{\di x}\bigg)\,\hat{u} = \hat{f}_i
                   \,\,\text{ in }Q, \quad i = 1,2,\dots,N,  \notag\\
&\hat{u}|_{t=0} =\hat{u}_0= u_0-w]_{t=0},  \quad \frac{\di \hat{u}}{\di t}\Big|_{t=0} = \hat{u}_1=
u_1-\frac{\di w}{\di t}\Big|_{t=0},  \quad   \hat{u}|_{S_T} = 0,
 \label{3.22a}
\end{align}
where
\begin{equation}  \label{4b}
\hat{f}_i=f_i -\frac{\di^2 w_i}{\di t^2}+
B_i\bigg(x,t,\frac {\di}{\di x}\bigg)\,w.
\end{equation}
Then
$$\hat{f} \in L^2(Q)^N, \,\, \hat{u}_0 \in  H^1_0(\Omega)^N, \,\, \, \hat{u}_1 \in L^2(\Omega)^N.  $$

It follows from  Theorem 4.2 that, there exists a unique solution to the problem \eqref{3.22a}  such that
\begin{equation}
\hat{u} \in L^2(0,T; H^1_0(\Omega))^N,
\quad \frac{\di \hat{u}}{\di t} \in L^2(Q)^N,
\label{5e}
\end{equation}
and
\begin{gather}
(\hat{f}, \hat{u}_0, \hat{u}_1) \to \,\,(\hat{u},  \frac{\di \hat{u}}{\di t} )  \,\,\text{is a linear continuous mapping of}
\notag \\
L^2(Q)^N \times H^1_0(\Omega)^N, \times L^2(\Omega)^N \text{ into }  L^2(0,T; H^1_0(\Omega))^N \times L^2(Q)^N.
\label{6e}
\end{gather}

Let $(\hat{f}_m, \hat{u}_{0m}, \hat{u}_{1m}) $ be a sequence such that
\begin{gather}
\hat{f}_m \in C^{\infty}([0,T],\mathcal{D}(\Omega))^N, \,\,  \hat{f}_m (x,t)= \sum_{k=0}^\infty\,\frac 1{k!}\,
                         \frac{\di^k  \hat{f}_m}{\di t^k}\, (x,0) t^k, \,\, \hat{f}_m \to \hat{f} \,\,\text{in} \,\, L^2(Q)^N, \notag \\
\hat{u}_{0m} \in \mathcal{D}(\Omega)^N,  \,\, \hat{u}_{0m} \to \hat{u}_0 -w|_{t=0} \quad\text{in }
 H^1_0(\Omega)^N,  \notag \\
 \hat{u}_{1m} \in \mathcal{D}(\Omega)^N,    \,\,  \hat{u}_{1m}    \to u_1-\frac{\di w}{\di t}\Big|_{t=0} \quad \text {in } L^2(\Omega)^N. \label{5b}
\end{gather}

Consider the problem: Find $\hat u_m$ satisfying
\begin{align}
&\frac{\di^2  \hat u_{mi}}{\di t^2} - B_i\bigg(x,t,\frac {\di}{\di x}\bigg)\,\hat u_m = \hat f_{mi}
                   \,\,\text{ in }\,\,Q, \quad i = 1,2,\dots,N,  \notag\\
&\hat u_m |_{t=0} =\hat u_{0m},  \quad \frac{\di \hat u_m}{\di t}\Big|_{t=0} =\hat u_{1m}.
 \label{6b}
\end{align}

It follows from Theorem 4.2  that,
 there exists the unique solution to
 the problem \eqref{6b} such that
$\hat{u}_m \in C^{\infty}(\overline{Q})^N$. This solution  is presented in the form
\begin{equation}\label{7b}
\hat{u}_m (x,t)=\hat u_{0m}(x) +   \hat u_{1m}(x)t +  \sum_{k=2}^\infty\,\frac 1{k!}\,
                         \frac{\di^k  \hat{u}_m}{\di t^k}\, (x,0) t^k, \,\,\, (x,t) \in Q,
\end{equation}
and by \eqref{5b}
\begin{equation}
\hat{u}_m \to \hat{u}\,\, \text{in} \,\,
 L^2(0,T; H^1_0(\Omega))^N, \,\,\frac{\di \hat{u}_m}{\di t} \to \frac{\di \hat{u}}{\di t} \,\,\text{in} \,\, L^2(Q)^N.
 \label{8.b}
\end{equation}
Thus, we have proved the following result:
\begin{theorem}
Let $\Omega$ be a bounded domain in ${\mathbb{R}}^n$ with a boundary $S$ of the class
$C^{\infty}$ and $T\in (0,\infty)$. Suppose that the conditions \eqref{2.47a}--\eqref{2.49a}
and \eqref{4f}  are satisfied.
Then there exists the unique solution to the problem \eqref{3.222} such that $u \in L^2(0,T; H^1(\Omega))^N$,
$\frac{\di u}{\di t} \in L^2(Q)^N $. This solution is presented in the form $u=\hat{u} +w$, where $w$ is defined in
\eqref{al} and $\hat{u}$ is determined by  \eqref{6b}   and     \eqref{8.b}.
\end{theorem}

\section{Problem on vibration of an orthotropic plate in a viscous medium.}

Plates fabricated from composite materials are used in modern constructions. Such plates
are orthotropic. The strain energy of the orthotropic plate is defined by the following formula,
see \cite{Lit.6}
\begin{equation}
\Phi(u) = \frac 12 \int_\Omega\, \bigg(D_1\bigg(\frac{\di^2 u}{\di x_1^2}\bigg)^2
                  +2\,D_{12} \,\frac{\di^2 u}{\di x_1^2}\,\frac{\di^2 u}{\di x_2^2}
                + D_2\bigg(\frac{\di^2 u}{\di x_2^2} \bigg)^2
                  + 2\,D_3  \bigg( \frac{\di^2 u}{\di x_1\,\di x_2} \bigg)^2\bigg)\,dx. \label{4.1}
\end{equation}
Here $\Omega$ is the midplane of the plate, $\Omega$  is a bounded domain in ${\mathbb{R}}^2$
with a boundary $S$,
\begin{equation}\notag
          dx = dx_1\, dx_2, \,\,
D_i = \frac {h^3 E_i}{12(1 - \mu_1\mu_2)}, \,\, i = 1,2, \,\, D_{12} = \mu_2 D_1
                           = \mu_1D_2, \,\,   D_3 = \frac {h^3 G}6,
\end{equation}
$E_1$, $E_2$, $G$, $\mu_1$, $\mu_2$ being the elasticity characteristics of the material, $h$ the
thickness of the plate,
\begin{equation}\label{C4.1}
\text{$E_1$, $E_2$, $G$   are positive constants, $\mu_1$  and $\mu_2$  are constants,  $0 \le \mu_i < 1$, $i =1,2$,}
\end{equation}
$u$ is the function of deflection, i.e.,   the function of displacements of points of the
midplane in the direction perpendicular  to the midplane.

We suppose that
\begin{equation}\label{4.3}
h \in C^{\infty} (\overline{\Omega}), \quad  e_1 \le h \le e_2, \quad \text{$e_1$, $e_2$ are positive constants}.
\end{equation}

Variation of the strain energy of the plate determines the following bilinear form
\begin{align}
a(u,v) = &\int_\Omega\, \Big[D_1\,\frac{\di^2 u}{\di x_1^2}\,\frac{\di^2 v}{\di x_1^2}
                  + D_2\,\frac{\di^2 u}{\di x_2^2}\,\frac{\di^2 v}{\di x_2^2}
                  + D_{12}\bigg( \frac{\di^2 u}{\di x_1^2}\,\frac{\di^2 v}{\di x_2^2}
                  + \frac{\di^2 u}{\di x_2^2}\,\frac{\di^2 v}{\di x_1^2}\bigg)   \notag\\
                  & + 2 D_3 \frac{\di^2 u}{\di x_1\,\di x_2} \,\frac{\di^2 v}{\di x_1\,\di x_2}\Big]\,dx. \label{4.4}
\end{align}
In our case  $a(u,u) = 2 \Phi(u)$.

We assume that the plate is clamped. Thus,
\begin{equation}\label{4.5}
u\big |_S = 0, \quad \frac{\di u}{\di \nu} \Big|_S = 0,
\end{equation}
where $\nu$ is the unit outward normal to $S$.

One can easily see that, on the set of smooth functions which satisfy the condition
\eqref{4.5}, the following equality holds.
\begin{equation}\label{4.6}
a(u,v) = (A u,v) = (u, A v).
\end{equation}
Here $(\cdot,\cdot)$ is the scalar product in $L^2(\Omega)$, and the operator
$A$ given as follows:
\begin{align}
A u  = &\frac{\di^2}{\di x_1^2}\bigg(D_1\,\frac{\di^2 u}{\di x_1^2}\bigg)
          + \frac{\di^2}{\di x_2^2}\bigg(D_2\,\frac{\di^2 u}{\di x_2^2}\bigg)
          + \frac{\di^2}{\di x_2^2}\bigg(D_{12}\,\frac{\di^2 u}{\di x_1^2}\bigg) \notag\\
       &+\frac{\di^2}{\di x_1^2}\bigg(D_{12}\,\frac{\di^2 u}{\di x_2^2}\bigg)
           + 2\, \frac{\di^2}{\di x_1\,\di x_2}\bigg(D_3\,\frac{\di^2 u}{\di x_1\,\di x_2}
           \bigg)=   - F_{\mathrm{re}},     \label{4.7}
\end{align}
$F_{\mathrm{re}} = - A u$ is the resistance force induced by the  elasticity for the function of displacement
$u$.

The viscous medium resists the vibration of the plate. The resistance  force $F_{\mathrm{rm}}$
that it induces is opposite in direction to the velocity $\frac{\di u}{\di t}$,
$F_{\mathrm{rm}} = -\varphi\,\frac{\di u}{\di t}$, where $\varphi$ is the resistance coefficient which
is an increasing function of $|\frac{\di u}{\di t}|$ that takes positive values.

We take the resistance force in the form
\begin{equation}\label{4.8}
F_{\mathrm{rm}} = - \bigg(a_0 + a_1\bigg(\frac{\di u}{\di t}\bigg)^2\bigg)\,\frac{\di u}{\di t},
\end{equation}
where $a_0$ and $a_1$ are positive constants.

The D'Alembert inertia force is given by
\begin{equation}\label{4.9}
F_{\mathrm{in}} = -  \rho h \,\frac{\di^2 u}{\di t^2},
\end{equation}
$\rho$ being the density,  a positive constant.

Let $K$ be an exterior transverse force that acts on the plate. According to the D'Alembert
principle, the sum of an active force that is applied at any point at each instant of time and
the internal and inertia forces  which it induces is equal to zero. Therefore,
\begin{equation}\label{4.10}
F_{\mathrm{re}} + F_{\mathrm{rm}} + F_{\mathrm{in}} + K = 0.
\end{equation}
From here, we obtain the following equation on vibration of the orthotropic  plate in a viscous
medium:
\begin{equation}\label{4.11}
 \rho h \,\frac{\di^2 u}{\di t^2} + A u + \bigg(a_0 + a_1\bigg(\frac{\di u}{\di t}\bigg)^2\bigg)\frac{\di u}{\di t} = K.
\end{equation}
Dividing  both sides of equation \eqref{4.11} by $\rho h$ gives
\begin{equation}\label{4.12}
\frac{\di^2 u}{\di t^2} + M u + \alpha_0 \, \frac{\di u}{\di t}  + \alpha_1\bigg(\frac{\di u}{\di t}\bigg)^2
           \frac{\di u}{\di t} = f,
\end{equation}
where
\begin{equation}\notag
M u = \frac 1{\rho h} \, A u,  \quad \alpha_0  = \frac{a_0}{\rho h}, \quad  \alpha_1 = \frac{a_1}{\rho h},
             \quad  f = \frac{K}{\rho h}.
\end{equation}

  According to      \eqref{4.5},   the boundary conditions have the form
  \begin{equation}\label{4.11a}
  u\big|_{S_T}=0, \quad \frac{\di u}{\di \nu}\Big|_{S_T}=0.
   \end{equation}
     We set the initial conditions in the form
\begin{equation}\label{4.13}
 u\big|_{t=0} = u _0, \quad   \frac{\di u}{\di t} \Big|_{t=0} = u_1.
 \end{equation}
We suppose
\begin{align}
&f \in  L^2 (Q), \quad \frac{\di f}{\di t} \in  L^2 (Q),  \,\, \text{i.e,}\,\, f \in H^{0,1}(Q), \label{4.14}\\
&u_0 \in  H_0^4(\Omega), \quad  u_1 \in  H_0^2(\Omega).\label{4.15}
\end{align}

\begin{theorem}
Let $\Omega$ be a bounded domain in ${\mathbb{R}}^2$ with a boundary $S$ of the class
$C^5$,  $T\in (0,\infty)$. Suppose that the conditions \eqref{C4.1}, \eqref{4.3},  \eqref{4.14},
\eqref{4.15}  are satisfied. Then there exists a unique
solution $u$ to the problem  \eqref{4.12},   \eqref{4.11a}, \eqref{4.13} such that $u \in W$, where
\begin{gather}
W = \Big\{v\mid v\in  L^\infty(0,T;H^4(\Omega)\cap H_0^2(\Omega)), \,\,
\frac{\di v}{\di t} \in  L^\infty(0,T;H_0^2(\Omega)), \notag\\
\frac{\di^2 v}{\di t^2}  \in  L^\infty(0,T;L^2(\Omega))\Big\},
\end{gather}                                                                                                     \label{5.17a}\\
and
\begin{gather}
(f, u_0, u_1) \mapsto{u} \,\, \text{is a continuous mapping of }\,\,  H^{0,1}(Q) \times H^4_0(\Omega)
\times  H^2_0(\Omega)  \,\, \text{into} \,\,W.\label{5.18a}
\end{gather}
Let $\{f_m,u_{0m},u_{1m},h_m\}$ be a sequence such that
\begin{align}
&f_m \in C^{\infty}([0,T];\mathcal{D}(\Omega)), \,\, f_m(x,t) = \sum_{k=0}^\infty\,\frac 1{k!}\,
  \frac{\di^k f_m}{\di t^k}\, (x,0) t^k,  \quad (x,t) \in \overline{Q}, \notag \\
&f_m\to f  \,\,\,\text{ in }\,\, H^{0,1}(Q),\quad u_{0m}\in \mathcal{D}(\Omega) \quad
     u_{0m} \to u_0 \,\,\,\text{ in }\,\, H^4_0(\Omega) \notag\\
&u_{1m} \in   \mathcal{D} (\Omega), \quad  u_{1m} \to u_1 \,\,\text{ in }\,\, H_0^2(\Omega), \notag\\
&h_m \in C^{\infty} (\overline{\Omega}), \quad e_1 \le h_m \le e_2, \quad
      h_m \to h \,\,\text{ in }\,\,C^3(\overline{\Omega}).     \label{4.188}
\end{align}
Let $u_m$ be the solution to the problem
\begin{align}
&\frac{\di^2 u_m}{\di t^2} + M_m u_m + \alpha_{0m}\frac {\di u_m}{\di t}
        + \alpha_{1m}\bigg(\frac {\di u_m}{\di t}\bigg)^2 \,\frac {\di u_m}{\di t} =f_m,  \notag\\
&u_m|_{S_T} = 0,  \quad     \frac {\di u_m}{\di \nu} \Big|_{S_T} = 0,  \label{4.199}
\end{align}
where $M_m = \frac 1{\rho h_m}A_m,\, \alpha_{0m}=\frac{a_0}{\rho h_m}, \,   \alpha_{1m}=\frac{a_1}{\rho h_m}
$, $A_m$  is defined by \eqref{4.7},
where $h$ is replaced by $h_m$.
Then
\begin{align}
& u_m \to u \,\,\text{ in }\,\, L^\infty(0,T;H^4(\Omega)\cap H_0^2(\Omega)), \notag\\
&\frac {\di u_m}{\di t}  \to \frac {\di u}{\di t} \,\,\text{ in }\,\, L^\infty(0,T;H_0^2(\Omega)), \notag\\
&\frac{\di^2 u_m}{\di t^2}  \to \frac{\di^2 u}{\di t^2} \,\,\text{ in }\,\, L^\infty(0,T;L^2(\Omega)).
                                                                          \label{4.200}
\end{align}
\end{theorem}

\begin{proof}
The existence of a unique solution $u$ to the problem \eqref{4.12}, \eqref{4.11a}, \eqref{4.13} such that $u\in W$
 and  \eqref{5.18a} holds     is proved by a small modification of the proofs of
 Theorem 2.1, Chapter 5 in  \cite{LiM.1} or Theorem 3.1, Chapter 1 in \cite{Lions}. In this case, we take into account that
\begin{equation}\notag
c\|u\|_{H_0^2(\Omega)}^2 \ge a(u,u) \ge c_1\|u\|_{H_0^2(\Omega)}^2, \,\,\, u \in H^2_0(\Omega),
\end{equation}
use the Faedo--Galerkin approximations, and the theorem on compactness, see Theorem 5.1,
Chapter 1 in \cite{Lions}, is applied to pass to the limit in the nonlinear term of  \eqref{4.12}.

Informally, the solution to the problem \eqref{4.199} is represented in the form
\begin{equation}\label{4.211}
u_m(x,t) = u_{0m}(x) + u_{1m}(x)t + \sum_{k=2}^\infty\,\frac 1{k!}\,
     \frac{\di^k u_m}{\di t^k}\, (x,0) t^k,  \quad (x,t) \in \overline{Q},
\end{equation}
where $\frac{\di^k u_m}{\di t^k}\, (x,0)$ are determined by the following recurrence relations
\begin{align}
&\frac{\di^k u_m}{\di t^k}\, (\cdot,0) = \frac{\di^{k-2} f_m}{\di t^{k-2}}\, (\cdot,0)
             - \frac{\di^{k-2} M_m u_m}{\di t^{k-2}}\, (\cdot,0)
             - \alpha_0  \frac{\di^{k-1} u_m}{\di t^{k-1}}\, (\cdot,0)  \notag\\
&           - \alpha_1 \sum_{j=0}^{k-2} \,C_{k-2}^j \Big[\frac{\di^j}{\di t^j}\bigg(\frac{\di u_m}
                 {\di t}\bigg)^2\Big](\cdot,0)\frac{\di^{k-j-1}u_m}{\di t^{k-j-1}}(\cdot,0), \quad k = 2,3,\dots      \notag
\end{align}

The convergence of the series \eqref{4.211}  is proved by analogy with the proof
 of Theorem 2.3. In this case, we consider the functions
 \begin{equation}
u_{me}(x,t) = u_{0m}(x) + u_{1m}(x)t + \sum_{k=2}^e\,\frac 1{k!}\,
     \frac{\di^k u_m}{\di t^k}\, (x,0) t^k,  \quad (x,t) \in \overline{Q}
\end{equation}
  and apply the infinite function theorem. Then we obtain that $ u_{me} \to u_m$ in $W$ as
 $   e \to \infty$.

 Since the solution to the problem    \eqref{4.12},   \eqref{4.11a}, \eqref{4.13}  depends
 continuously on the data of the problem,   \eqref{4.200} follows from   \eqref{4.188}.
\end{proof}

\section{Maxwell's equations.}
\subsection{General problem.}

We consider the following  problem of electromagnetism: Find functions $D$ and $B$
such that, see \cite{Du.}, \cite{Mau.}
\begin{align}
&\frac{\di D}{\di t} - \curl(\hat{\mu}B) + \sigma \hat{\xi} D = G_1 \quad\text{in } Q,
   \label{6.1} \\
&\frac{\di B}{\di t} + \curl(\hat{\xi} D) = G_2 \quad\text{in } Q, \label{6.2}  \\
 &\nu\wedge D  = 0\quad\text{on } S_T, \label{6.3}  \\
 &D\Big|_{t=0} = D_0, \quad B\Big|_{t=0} = B_0 \quad\text{in } \Omega.
      \label{6.4}
\end{align}
Here $Q= \Omega\times (0,T)$, $T < \infty$, $\Omega$ is a bounded domain in $\mathbb{R}^3$ with a
boundary $S$, $S_T = S \times (0,T)$, $D$ is the electric induction, $B$  is the magnetic
induction, $\hat{\mu}$, $\hat{\xi}$, and $\sigma$  are scalar functions of $x\in \Omega$ that take positive
values, $\nu$ is the unit outward normal to $S$.

We define the following spaces
\begin{align}
&V = \{v\mid v \in L^2(\Omega)^3, \,\,\curl v \in L^2(\Omega)^3\},  \notag \\
&V_1 = \{v\mid v \in V, \,\, v\wedge \nu = 0\}.      \label{6.5}
\end{align}
The space $V_1$ is the closure of $\mathcal{D}(\Omega)^3$ with respect to the norm of $V$,
\begin{equation}\label{6.6}
\|v\|_V  = \bigg(\|v\|_{L^2(\Omega)^3}^2 + \|\curl v\|_{L^2(\Omega)^3}^2\bigg)^{\frac  12}.
\end{equation}
For further detail about the spaces $V$ and$V_1$, see \cite{Gir.}, Chapter 1, Sections 2,3, and \cite{Du.}, Chapter~7.

Let also
\begin{align}
&X = \left\{h\mid h \in L^\infty(0,T;V), \,\, \frac{\di h}{\di t} \in L^\infty(0,T;L^2(\Omega)^3)\right\},  \notag \\
&X_1 = \left\{h\mid h \in L^\infty(0,T;V_1), \,\, \frac{\di h}{\di t} \in L^\infty(0,T;L^2(\Omega)^3)\right\}.     \notag
\end{align}
The norm in $X$ and $X_1$ is defined by
\begin{equation}\notag
\|h\|_X  = \|h\|_{L^\infty(0,T;V)} + \left\|\frac{\di h}{\di t}\right\|_{ L^\infty(0,T;L^2(\Omega)^3)}.
\end{equation}
We suppose
\begin{gather}
G_1 \in H^{0,1}(Q)^3, \quad G_2 \in H^{0,1}(Q)^3, \quad D_0 \in V_1, \,\,B_0 \in V, \label{6.6b}\\
\hat{\mu} \in C^1(\overline{\Omega}), \quad \mu_1\ge \hat{\mu}\ge \mu_2, \quad
\hat{\xi} \in C^1 (\overline{\Omega}), \quad \xi_1\ge \hat{\xi} \ge \xi_2,  \label{6.6c}\\
\sigma \in L^\infty(\Omega), \quad \tilde{\sigma}_1 \ge \sigma \ge \tilde{\sigma}_2. \label{6.6d}
\end{gather}
Here $\mu_1$, $\mu_2$, $\xi_1$, $\xi_2$, $\tilde{\sigma}_1$, $\tilde{\sigma}_2$ are positive constants.

\begin{theorem}
Let $\Omega$ be a bounded domain in $\mathbb{R}^3$ with a boundary $S$ of the class $C^\infty$.
Suppose that the conditions \eqref{6.6b}--\eqref{6.6d} are satisfied. Then, there exists a unique solution
to the problem \eqref{6.1}--\eqref{6.4} such that.
\begin{equation}\label{6.7}
D \in X_1, \quad B \in X.
\end{equation}
\end{theorem}
Theorem 6.1 is proved in \cite{Du.}, Chapter 7, by using Galerkin approximations.

Let us discuss construction of the solution to the problem \eqref{6.1}--\eqref{6.4}. In order to apply our method
to this problem, we should somewhat change the formulation of this problem.

We present the function $B$ in the form
\begin{equation}\label{6.8a}
B = B^1 + B^2, \quad B^1 \in X_1, \quad B^2 \in X, \quad \frac{\di^2 B^2}{\di t^2} \in L^2(Q)^3.
\end{equation}
We consider  that $B^1$ is unknown, while $B^2$ is  given and satisfies the condition
\begin{equation}\label{6.8b}
B^2\wedge \nu = B\wedge \nu \quad \text{in  } \,\,  L^\infty(0,T;H^{-\frac12}(S)^3).
\end{equation}
Here $B$ is the solution  to the problem \eqref{6.1}--\eqref{6.4} together with $D$. Equality \eqref{6.8b}
has sense for elements of $X$, see  \cite{Du.}, Lemma 4.2, Chapter 7.

According to \eqref{6.8a}, \eqref{6.8b}, we set
\begin{align}
&B_0 = B_0^1 + B_0^2, \quad B_0^1 \in V_1, \quad B_0^2 \in V, \notag\\
&B^1|_{t=0} = B_0^1,\quad B^2|_{t=0} = B_0^2.  \label{6.8c}
\end{align}
Now the problem \eqref{6.1}--\eqref{6.4} is represented as follows:
\begin{align}
&\frac{\di D}{\di t} - \curl(\hat{\mu}B^1) + \sigma \hat{\xi} D  = G_1 + \curl(\hat{\mu}B^2)\quad \text{in } Q,     \label{6.10f} \\
&\frac{\di B^1}{\di t} + \curl(\hat{\xi} D) = G_2 - \frac{\di B^2}{\di t}\quad \text{in }Q,      \label{6.11f}  \\
&\nu\wedge D  = 0\quad\text{on } S_T, \,\,\nu\wedge B^1  = 0\quad\text{on }S_T,
                                     \label{6.12f}  \\
&D\Big|_{t=0} = D_0, \quad B^1\Big|_{t=0} = B_0^1 \quad\text{in } \Omega  \label{6.13f}.
\end{align}
The existence of a unique solution $(D,B^1)$ to the problem \eqref{6.10f}--\eqref{6.13f} such that
$D\in X_1$,  $B^1 \in X_1$ follows from Theorem 6.1.

Thus, if the pair $(D,B)$  is the solution to the problem \eqref{6.1}--\eqref{6.4}, and
\begin{equation}
B^2 \in X, \quad  \frac{\di^2 B^2}{\di t^2} \in L^2(Q)^3, \label{6.13p}
\end{equation}
and \eqref{6.8b} is satisfied, then the pair $(D,B^1)$  with $B^1= B-B^2$ is the solution to the problem
   \eqref{6.10f}--\eqref{6.13f}.

   On the contrary, if the couple  $(D,B^1)$  is the solution to the problem \eqref{6.10f}--\eqref{6.13f}, where $B^2$
   meets  \eqref{6.13p}, then the couple $(D, B)$  with $B=B^1+B^2$  is the solution to the problem
   \eqref{6.1}--\eqref{6.4}, and  \eqref{6.8b}   holds.

Therefore, the formulations \eqref{6.1}--\eqref{6.4} and  \eqref{6.10f}--\eqref{6.13f} are equivalent in the
above sense.

\medskip
Let $\{G_{1n},\, G_{2n},\,D_{0n},\,B_{0n}^1,\,\hat{\mu}_n,\,\hat{\xi}_n,\,\sigma_n\}$ be a sequence such
that
\begin{gather}
G_{in} \in C^{\infty}([0,T];\mathcal{D}(\Omega)^3),    \notag \\
G_{in}(x,t) =  \sum_{k=0}^\infty\,\frac 1{k!}\,
     \frac{\di^k G_{in}}{\di t^k}\, (x,0) t^k,  \quad (x,t) \in Q,\quad i=1,2,  \notag \\
G_{1n} \to G_1 + \curl(\hat{\mu}B^2 )  \quad\text{in}\,\, H^{0,1}(Q)^3,   \notag \\
G_{2n} \to G_2 - \frac{\di B^2}{\di t} \quad\text{in}\,\, H^{0,1}(Q)^3,    \label{6.1a}\\
D_{0n} \in \mathcal{D}(\Omega)^3, \,\, D_{0n} \to D_0 \,\,\text{in}\,\, V_1, \,\,
 B^1_{0n} \in \mathcal{D}(\Omega)^3, \,\,
  B^1_{0n} \to B^1_0 \,\,\text{in}\,\, V_1,
\label{6.1b} \\
   \hat{\mu}_n\in C^\infty(\overline{\Omega}),\,  \hat{\mu}_n\to   \hat{\mu} \,\text{in } C^1(\overline{\Omega}), \,\,                            \hat{\xi}_n\in C^\infty(\overline{\Omega}),\,  \hat{\xi}_n\to \hat{\xi} \quad\text{in } C^1(\overline{\Omega}), \notag \\
\sigma_n \in C^\infty(\overline{\Omega}),  \,\,\sigma_n \to \sigma\quad\text{in}\,\, L^\infty(\Omega).
 \label{6.1c}
\end{gather}

Consider the problem: Find functions $D_n$ and $B_n^1$ such that
\begin{align}
&\frac{\di D_n}{\di t} - \curl(\hat{\mu}_n B_n^1) + \sigma_n \hat{\xi}_n D_n  =
                G_{1n}   \quad\text{in }Q,     \label{6.9a} \\
&\frac{\di B_n^1}{\di t} + \curl(\hat{\xi}_n D_n) = G_{2n} \quad \text{in } Q,\label{6.9b}  \\
&\nu\wedge D_n  = 0\quad\text{on } S_T, \,\,\nu\wedge B_n^1  = 0\quad\text{on } S_T,
                                     \label{6.9c}  \\
&D_n\Big|_{t=0} = D_{0n}, \quad B_n^1\Big|_{t=0} = B_{0n}^1 \quad\text{in } \Omega \label{6.9d}.
\end{align}

\begin{theorem}
Let $\Omega$ be a bounded domain, in $\mathbb{R}^3$ with a boundary $S$ of the class $C^\infty$
and $T \in (0,\infty)$.
Suppose that the conditions of Theorem 6.1 and \eqref{6.8a}, \eqref{6.8b},  \eqref{6.8c} are satisfied.
Let also \eqref{6.1a}--\eqref{6.1c} hold. Then for any $n \in \mathbb{N}$  there exists a unique solution
$(D_n,B_n^1)$ to the problem \eqref{6.9a}--\eqref{6.9d} that is represented in the form

\begin{gather}
D_n(x,t)  = D_{0n}(x) +  \sum_{k=1}^\infty\,\frac 1{k!}\,
     \frac{\di^k D_n}{\di t^k}\, (x,0) t^k,   \label{6.10}\\
B_n^1(x,t)  = B_{0n}^1(x) +  \sum_{k=1}^\infty\,\frac 1{k!}\,
     \frac{\di^k B^1_n}{\di t^k}\, (x,0) t^k,   \label{6.11}
\end{gather}
where
\begin{align}
&\frac{\di^k  D_n}{\di t^k}(x,0) = \curl\bigg(\hat{\mu}_n(x) \frac{\di^{k-1} B_n^1}{\di t^{k-1}}\bigg)(x,0)
     - \sigma_n (x)\hat{\xi}_n (x) \frac{\di^{k-1}D_n}{\di t^{k-1}}(x,0)            \notag\\
&  + \frac{\di^{k-1} G_{1n}}{\di t^{k-1}}(x,0), \quad k = 1,2,\dots,      \notag \\
&\frac{\di^k B^1_n}{\di t^k}(x,0) = - \curl\bigg(\hat{\xi}_n(x) \frac{\di^{k-1}D_n}{\di t^{k-1}}\bigg)(x,0)
   + \frac{\di^{k-1} G_{2n}}{\di t^{k-1}}(x,0), \quad x \in \Omega, \quad k = 1,2,\dots.
       \label{6.12}
\end{align}
The series for $D_n$ and $B_n^1$ converge in $X_1$ and
\begin{equation}\label{6.13}
D_n \to  D \quad\text{in } X_1, \quad B_n^1 \to B^1\quad \text{in } X_1,
\end{equation}
where $(D,B^1)$ is the solution to the problem \eqref{6.10f}--\eqref{6.13f}.
\end{theorem}
\begin{proof}
The existence of the unique solution to the problem \eqref{6.9a}--\eqref{6.9d} follows from
 Theorem 6.1. The condition of compatibility  of order infinity for this problem is satisfied.
Because of this, informally, the solution  to the problem \eqref{6.9a}--\eqref{6.9d} is represented
in the form \eqref{6.10}, \eqref{6.11} .

It follows from the proofs of Theorems 5.1 and 4.1 in \cite{Du.}, Chapter 7 that, in
the case where $\hat{\xi}$, $\hat{\mu}$, and $\sigma$  are fixed functions that satisfy conditions \eqref{6.6c}, \eqref{6.6d}, the following inequality for the solution
to the problem \eqref{6.10f}--\eqref{6.13f} holds:
\begin{equation}\label{6.14}
\|D\|_{X_1} + \|B^1\|_ {X_1} \le  C(\|G_1\|_{H^{0,1}(Q)^3} + \|G_2\|_{H^{0,1}(Q)^3} +
\|D_0\|_{V_1} + \|B_0^1\|_ {V_1}),
\end{equation}
where $C$ depends on $\hat{\xi}$, $\hat{\mu}$, and $\sigma$.

The converges of the series  \eqref{6.10} and \eqref{6.11} in $X_1$ is proved analogously
to  the above by  using  \eqref{6.1a}--\eqref{6.1c}, and \eqref{6.14}.

Taking  \eqref{6.1a}--\eqref{6.1c} into account in the same way  as it is done in \cite{Du.},
Theorems 4.1 and~5.1, Chapter 7, we get
\begin{equation}\label{6.15}
\|D_n\|_{X_1} \le C_1,  \quad  \|B_n^1\|_ {X_1} \le  C_2.
\end{equation}
Therefore, we can extract a subsequence $\{D_m, B_m^1\}$ such that
\begin{align}
&D_m \to D\quad  \text{$*$-weakly in } X_1, \notag\\
&B_m^1 \to B^1\quad  \text{$*$-weakly in }X_1.        \label{6.16}
\end{align}
Let $w_1$ and $w$ be arbitrary elements of $L^2(\Omega)^3$. We take the scalar products
of \eqref{6.9a} and \eqref{6.9b} for $n=m$ with $w_1$ and $w$, respectively, in  $L^2(\Omega)^3$.
This gives
\begin{align}
&\bigg(\frac{\di D_m}{\di t}, w_1\bigg) -\Big (\curl(\hat{\mu}_m B_m^1), w_1\bigg)
          + (\sigma_m\hat{\xi}_m D_m, w_1) = (G_{1m}, w_1) \quad\text{a.e. in }  (0,T),
            \label{6.17a} \\
&\bigg(\frac{\di B_m^1}{\di t}, w\bigg) + \Big (\curl(\hat{\xi}_m D_m), w\bigg) = (G_{2m}, w)\quad\text{a.e. in } (0,T).    \label{6.17b}
\end{align}

Taking \eqref{6.1a}--\eqref{6.1c}  and \eqref{6.16} into account, we pass to the limit as $m \to \infty$
in  \eqref{6.17a},  \eqref{6.17b}, and \eqref{6.9c}, \eqref{6.9d}. We conclude that the pair $(D,B^1)$
determined in \eqref{6.16} is a solution to the problem \eqref{6.10f}--\eqref{6.13f}. Since the solution
to this problem is unique in $X_1 \times X_1$, \eqref{6.16} is also valid when $m$ is replaced by  $n$.

It remains to prove \eqref{6.13}.

We subtract equalities  \eqref{6.9a}--\eqref{6.9d} from \eqref{6.10f}--\eqref{6.13f}, respectively. This gives
\begin{align}
&\frac{\di }{\di t}(D-D_n) - \curl(\hat{\mu} B^1 - \hat{\mu}_n B_n^1)
          + \sigma\hat{\xi} D - \sigma_n \hat{\xi}_n D_n  = G_1 + \curl(\hat{\mu} B^2) -  G_{1n},
            \label{6.18} \\
&\frac{\di}{\di t}(B^1-B_n^1) + \curl(\hat{\xi} D - \hat{\xi}_n D_n) = G_2 - \frac{\di B^2}{\di t} - G_{2n},
           \label{6.19} \\
&\nu\wedge(D-D_n) = 0 \,\,\text{on } \,\,  S_T, \quad \nu\wedge(B^1-B_n^1) = 0 \quad\text{on }   S_T,
             \label{6.20} \\
&(D-D_n)\Big|_{t=0} = D_0-D_{0n} \quad\text{in }  \Omega, \quad (B^1-B_n^1)\Big|_{t=0} = B_0^1-B_{0n}^1
                   \quad\text{in }   \Omega.   \label{6.21}
\end{align}

  We have
 \begin{align}
&\curl(\hat{\mu} B^1 - \hat{\mu}_n B_n^1) = \curl(\hat{\mu}( B^1 - B_n^1))
    + \curl((\hat{\mu}  - \hat{\mu}_n) B_n^1),    \notag    \\
&\sigma\hat{\xi} D - \sigma_n \hat{\xi}_n D_n  = \sigma\hat{\xi} (D - D_n)
    + D_n (\sigma\hat{\xi}  - \sigma_n \hat{\xi}_n),    \notag    \\
&\curl(\hat{\xi} D - \hat{\xi}_n D_n) =  \curl(\hat{\xi}( D - D_n))
    + \curl((\hat{\xi} -\hat{\xi}_n) D_n).       \label{6.22}
\end{align}

 We denote
  \begin{align}
&\gamma_{1n} = -\curl((\hat{\mu}  - \hat{\mu}_n) B_n^1)
       +  D_n (\sigma\hat{\xi}  - \sigma_n \hat{\xi}_n),    \notag    \\
&\gamma_{2n} = \curl((\hat{\xi} - \hat{\xi}_n) D_n). \label{6.23}
\end{align}

 \eqref{6.1c} and \eqref{6.15}  yield
 \begin{equation}\label{6.24}
\gamma_{1n} \to 0  \quad\text{in} \,\,  L^{\infty}(0,T; L_2(\Omega)^3),  \,\,
\gamma_{2n} \to 0  \quad\text{in} \,\, L^{\infty}(0,T; L_2(\Omega)^3).
\end{equation}

By \eqref{6.22}--\eqref{6.24} equations \eqref{6.18}, \eqref{6.19} take the form
\begin{align}
&\frac{\di }{\di t}(D-D_n) - \curl(\hat{\mu} (B^1 -  B_n^1))
          + \sigma\hat{\xi} (D - D_n) +\gamma_{1n}  = G_1 +  \curl(\hat{\mu} B^2) -  G_{1n}
             \notag \\
&\frac{\di}{\di t}(B^1-B_n^1) + \curl(\hat{\xi}( D -  D_n))+\gamma_{2n}.
         = G_2 - \frac{\di B^2}{\di t} -G_{2n}.  \notag
\end{align}
From here and \eqref{6.14},  taking \eqref{6.1a}, \eqref{6.1b}, and \eqref{6.24}
into account, we obtain \eqref{6.13}.
\end{proof}

According to the theory of electromagnetism, the function $B$ should satisfy the condition
\begin{equation}\label{6.26}
\diver B = 0 \quad\text{in } Q.
\end{equation}
\begin{theorem}
Suppose that the conditions of Theorem 6.1 are satisfied and, in addition,
\begin{equation}\label{6.27}
\diver G_2 =  0 \quad\text{in } Q, \quad \diver B_0 =  0 \quad\text{in } \Omega.
\end{equation}
Then the function $B$  of the solution $(D,B)$ to the problem \eqref{6.1}--\eqref{6.4}
also meets the condition $\diver B =  0$ in  $Q$.
\end{theorem}

Indeed, applying the operator $\diver$, in the sense of distributions, to both sides of
equation~\eqref{6.2}, we obtain
\begin{equation}\notag
\frac{\di}{\di t}(\diver B) = \diver G_2 \quad\text{in } Q.
\end{equation}
That is
\begin{equation}\notag
\diver B(\cdot,t) = \diver B_0 + \int_0^t \, \diver G_2(\cdot,\tau)d\tau = 0 \quad\text{in }(0,T).
\end{equation}

\begin{theorem}
Suppose that the conditions of Theorem 6.2 are satisfied and, in addition,
\begin{align}
&G_2 = \curl F, \quad F \in L^2(0,T;V), \quad \frac{\di F}{\di t} \in L^2(0,T;V), \notag\\
&B^2 = \curl P, \,\, P \in L^\infty(0,T;H^2(\Omega)^3), \,\, \frac{\di P}{\di t}
                     \in L^\infty(0,T;H^1(\Omega)^3), \notag\\
&\frac{\di^2 P}{\di t^2} \in L^2(0,T;H^1(\Omega)^3), \notag\\
&B_0^1 = \curl M^1, \,\, M^1\in  H_0^2(\Omega)^3, \,\,
                  B_0^2 = \curl M^2, \,\, M^2\in  H^2(\Omega)^3, \,\,
P|_{t=0} = M^2.                           \label{6.28k}
\end{align}

The corresponding functions $G_{2n}$ and $B_{0n}^1$ are given as follows
\begin{align}
&G_{2n} = \curl F_n, \quad F_n \in C^\infty([0,T];\mathcal{D}(\Omega)^3), \notag\\
& F_n(x,t) = \sum_{k=0}^\infty\,\frac 1{k!}\,
             \frac{\di^k  F_n}{\di t^k}\, (x,0) t^k,  \quad (x,t) \in Q,  \label{6.29} \\
& \curl F_n \to  \curl\bigg(F - \frac{\di P}{\di t}\bigg) \quad\text{in } L^2(Q)^3,  \notag\\
&\curl \frac{\di F_n}{\di t} \to \curl \bigg(\frac{\di F}{\di t} - \frac{\di^2 P}{\di t^2}\bigg)
                \quad\text{in } L^2(Q)^3,    \label{6.30} \\
&B_{0n}^1 = \curl M_n^1, \quad M_n^1\in \mathcal{D} (\Omega)^3, \quad
                 M_n^1 \to M^1  \quad\text{in }H_0^2(\Omega)^3.  \label{6.31}
\end{align}
Then the solution $D_n$, $B_n^1$ to the problem \eqref{6.9a}--\eqref{6.9d}
also meets the condition $\diver B_n^1 = 0$,
\eqref{6.13} holds and $\diver B = 0$.
\end{theorem}

Theorem 6.4 follows from results of Theorems 6.2 and 6.3.
\subsection{Slotted antenna}

We consider the problem on diffraction of electromagnetic  wave by a superconductor, see
\cite{Du.}, Chapter 7, Section 3.4. Let $\Omega_1$ be a bounded domain in $\mathbb{R}^3$,
of a superconductor, the boundary $S$ of $\Omega_1$  is of the class $C^\infty$. We consider a problem
in a domain $\Omega$ in $\mathbb{R}^3$ with an internal boundary $S$.
We assume that  $\Omega$ is a bounded domain.

We seek a solution to the following problem: Find vector functions $D$ and $B$ such that
\begin{align}
&\frac{\di D}{\di t} - \curl(\hat{\mu}B) + \sigma \hat{\xi} D  = G_1\quad\text{in } Q,  \label{6.28f} \\
&\frac{\di B}{\di t} + \curl(\hat{\xi} D) = G_2\quad\text{in } Q,      \label{6.29f}  \\
&\diver D = 0 \quad\text{in } Q, \quad \nu\wedge D  = 0\quad\text{on } S_T,   \label{6.30f}  \\
&\diver B = 0 \quad\text{in } Q, \quad \nu\cdot B  = 0\quad\text{on } S_T,   \label{6.31f}  \\
&D\Big|_{t=0} = D_0, \quad B\Big|_{t=0} = B_0 \quad\text{in } \Omega.  \label{6.32f}
\end{align}
We introduce the following spaces:
\begin{align}
&X_2 = \left\{h\mid h = \curl w, w \in L^2(0,T;V), \quad \frac{\di w}{\di t} \in L^2(0,T;V)\right\}, \notag\\
&X_3 = \bigg\{h\mid h = \curl w, w \in L^2(0,T;H^1(\Omega)^3), \label{6.32} \\
& h\cdot \nu = 0\ \text{in }  L^2(0,T;H^{-\frac12}(S)),\
              \frac{\di w}{\di t} \in L^2(0,T;H^1(\Omega)^3)\bigg\}. \label{6.33}
\end{align}
We assume
\begin{align}
&G_1 = \curl w \in X_2, \quad  G_2 = \curl u\in X_3,\quad D_0 =\curl p, \quad p\in H_0^2(\Omega)^3,
              \notag \\
&B_0 =\curl v,  \quad  v\in H^2(\Omega)^3,  \quad   \curl v\cdot \nu = 0  \quad\text{on }  S.  \label{6.37}
 \end{align}

\begin{theorem}
Let $\Omega$ be a bounded domain in $\mathbb{R}^3$ with a boundary $S$ of the class $C^\infty$.
Suppose that the conditions \eqref{6.37} are satisfied. Let also $\hat{\xi}$, $\hat{\mu}$, $\sigma$
be positive constants. Then, there exists a unique solution to the problem \eqref{6.28f}--\eqref{6.32f}
 such that
\begin{align}
&D \in L^\infty(0,T;H_0^1(\Omega)^3), \quad \frac{\di D}{\di t}
                     \in L^\infty(0,T;L^2(\Omega)^3), \notag\\
&B \in L^\infty(0,T;H^1(\Omega)^3), \quad \frac{\di B}{\di t}
                     \in L^\infty(0,T;L^2(\Omega)^3).   \label{6.38}
\end{align}
\end{theorem}

Indeed, the existence of a unique solution to the problem   \eqref{6.28f},  \eqref{6.29f}, \eqref{6.31f},
such that $\nu \wedge D = 0$ on $S_T$ and $\diver B = 0$ in $Q$,  follows from
 Theorems 6.1 and 6.3.  The conditions $\nu\cdot B = 0$  on $S_T$, $\diver D = 0$ in $Q$,
 and \eqref{6.38} follow from  Theorems 5.3, 6.3 and 6.4 in \cite{Du.}, Chapter 7.

As before, we represent the function $B$ in the form $B=B^1 + B^2$, where $B^2$ is a given
 function such that

\begin{align}
&B^2 =\curl \alpha^2, \quad \alpha^2 \in L^\infty(0,T;H^2(\Omega)^3), \quad
            \curl \alpha^2 \cdot \nu = 0 \quad\text{on }  S_T,  \notag\\
&\frac{\di \alpha^2 }{\di t} \in L_{\infty} (0,T;H^1(\Omega)^3), \quad
    \frac{\di^2 \alpha^2 }{\di t^2} \in L^2 (0,T;H^1(\Omega)^3),  \notag\\
&\curl \frac{\di \alpha^2 }{\di t} \cdot \nu = 0 \quad\text{on }  S_T. \label{6.39}
\end{align}

Let $B_\tau$ be the tangential component of the vector $B$ on $S_T$. It is determined as
$B_\tau=B|_{S_T} - B \cdot \nu$. Since  $ B \cdot \nu = 0$.
we get $B|_{S_T} = B_\tau$ and the following boundary condition for $B^2$:
\begin{equation}\label{6.40}
B^2|_{S_T} = B|_{S_T} \quad\text{in } H^{\frac12}(S_T)^3.
\end{equation}
According to \eqref{6.39}, \eqref{6.40}, we set
\begin{align}
&B_0 = B_0^1 + B_0^2, \quad  B_0^2 = B^2|_{t=0} = \curl \alpha^2|_{t=0} \in H^1(\Omega)^3,
                          \notag\\
&B_0^1 = \curl v - \curl \alpha^2|_{t=0} \in H_0^1(\Omega)^3).    \label{6.41}
\end{align}

Now for the functions $D$, $B^1$, we obtain the following problem:
\begin{align}
&\frac{\di D}{\di t} - \curl(\hat{\mu}B^1) + \sigma \hat{\xi} D  =
                      G_1 + \curl(\hat{\mu}B^2)\quad\text{in } Q,      \notag\\
&\frac{\di B^1}{\di t} + \curl(\hat{\xi} D) =
                      G_2  - \frac{\di B^2}{\di t}  \quad\text{in } Q,      \notag\\
&\diver D = 0 \quad\text{in } Q, \quad \nu\wedge D = 0\quad\text{on } S_T,   \notag\\
&\diver B^1 = 0 \quad\text{in } Q, \quad \nu\wedge B^1  = 0\quad\text{on } S_T,   \notag \\
&D\Big|_{t=0} = D_0, \quad B^1\Big|_{t=0} = B_0^1 = B_0 - B_0^2 \quad\text{in } \Omega.  \label{6.42}
\end{align}
By analogy with the above, we get the next result.
\begin{theorem}
Let $\Omega$ be a bounded domain in $\mathbb{R}^3$ with a boundary $S$ of the class $C^\infty$.
Suppose that the conditions \eqref{6.37} and  \eqref{6.39}--\eqref{6.41} are satisfied. Let also
$\hat{\xi}$, $\hat{\mu}$, $\sigma$ be positive constants. Then, there exists a unique solution
$(D, B^1)$  to the problem \eqref{6.42} that meets the conditions
\begin{align}
&D \in L^\infty(0,T;H_0^1(\Omega)^3), \quad \frac{\di D}{\di t}
                     \in L^\infty(0,T;L^2(\Omega)^3), \notag\\
&B^1 \in L^\infty(0,T;H_0^1(\Omega)^3), \quad \frac{\di B^1}{\di t}
                     \in L^\infty(0,T;L^2(\Omega)^3).   \label{6.43}
\end{align}
\end{theorem}

Thus, if the pair $(D,B)$  is the solution to the problem \eqref{6.28f}--\eqref{6.31f}, and $B^2$
meets \eqref{6.39}, \eqref{6.40},
then the pair $(D,B^1)$  with $B^1= B-B^2$ is the solution to the problem
   \eqref{6.42}.

 On the contrary, if the couple  $(D,B^1)$  is the solution to the problem \eqref{6.42}, where $B^2$
   meets  \eqref{6.39}, then the couple $(D, B)$  with $B=B^1+B^2$  is the solution to the problem
   \eqref{6.28f}--\eqref{6.31f}, and  \eqref{6.40}   holds.

Therefore, the formulations \eqref{6.1}--\eqref{6.4} and  \eqref{6.10f}--\eqref{6.13f} are equivalent in the
above sense.

Let $\{G_{1n},\, G_{2n},\,D_{0n},\,B_{0n}^1\}$  be a sequence such that
\begin{align}
&G_{in} = \curl w_{in},\quad w_{in} \in C^{\infty}([0,T];\mathcal{D}(\Omega)^3),    \notag \\
&w_{in}(x,t) =  \sum_{k=0}^\infty\,\frac 1{k!}\,
     \frac{\di^k w_{in}}{\di t^k}\, (x,0) t^k,  \quad (x,t) \in Q,\quad i=1,2,  \notag \\
&G_{1n} \to G_1 + \curl(\hat{\mu}B^2 )  \quad\text{in } H^{0,1}(Q)^3,   \notag \\
&G_{2n} \to G_2 - \frac{\di B^2}{\di t} \quad\text{in } H^{0,1}(Q)^3,  \label{6.44}\\
&D_{0n} = \curl p_n,\,\, p_n\in \mathcal{D}(\Omega)^3. \,\,
                                     \curl p_n \to \curl p \quad\text{in } H_0^1(Q)^3,  \notag \\
&B_{0n}^1 = \curl e_n,\,\, e_n\in \mathcal{D}(\Omega)^3, \,\,
                 \curl e_n \to \curl v - \curl \alpha^2|_{t=0} \in H_0^1(\Omega)^3.\label{6.45}
\end{align}

We consider the problem: Find functions $D_n$ and $B_n^1$ such that
\begin{align}
&\frac{\di D_n}{\di t} - \curl(\hat{\mu}B_n^1) + \sigma \hat{\xi} D_n  =
                                                  G_{1n}\quad\text{in } Q,  \notag\\
&\frac{\di B_n^1}{\di t} + \curl(\hat{\xi} D_n) = G_{2n}\quad\text{in } Q,      \notag\\
&\diver D_n = 0 \quad\text{in } Q, \quad \nu\wedge D_n  = 0\quad\text{on } S_T,
                                                                                          \notag  \\
&\diver B_n^1 = 0 \quad\text{in } Q, \quad \nu\wedge B_n^1  = 0\quad\text{on } S_T,
                                                                                       \notag  \\
&D_n\Big|_{t=0} = D_{0n}, \quad B_n^1\Big|_{t=0} = B_{0n}^1 \quad\text{in } \Omega.
                       \label{6.46}
\end{align}

\begin{theorem}
Let $\Omega$ be a bounded domain in $\mathbb{R}^3$ with a boundary $S$ of the class $C^\infty$.
Suppose that the conditions \eqref{6.44}, \eqref{6.45} are satisfied, and let
 $ \hat{\xi}$, $\hat{\mu}$, $\sigma $ be positive constants.
 Then for any $n \in \mathbb{N}$,  there exists a unique solution $D_n$, $B_n^1$
 to the problem \eqref{6.46} that is represented in the form
 \eqref{6.10}--\eqref{6.12} and
\begin{align}
&D_n \to D  \quad\text{in }  L^\infty(0,T;H_0^1(\Omega)^3), \quad
                     \frac{\di D_n}{\di t}  \to   \frac{\di D}{\di t} \quad\text{in }
                     L^\infty(0,T;L^2(\Omega)^3),  \notag\\
&B_n^1 \to B^1  \quad\text{in }  L^\infty(0,T;H_0^1(\Omega)^3), \quad
                     \frac{\di B_n^1}{\di t}  \to   \frac{\di B^1}{\di t} \quad\text{in }
                     L^\infty(0,T;L^2 (\Omega)^3).
 \label{6.47}
\end{align}
\end{theorem}
The  proof of this theorem is analogous to the proof of Theorem 6.2.

 \begin{remark}
   The problem \eqref{6.28f}
   --\eqref{6.31f} is connected with finding functions $y$ such that
  $\diver y =0$ in $\Omega$, $y\cdot\nu=0$ on   $S$, see  \eqref{6.33}, \eqref{6.37}.
  These functions can be determined
   in the form
   \begin{equation}
    y=\curl v + \grad h, \,\, v \in H^1(\Omega)^3, \,\, h\in H^1(\Omega),
   \end{equation}
  where $h$ is the solution to the problem
  \begin{gather}
  \diver \grad h = \Delta h =0  \notag\\
  \grad h\cdot\nu = \frac{\di h}{\di \nu}\Big|_S=-\curl v\cdot \nu.
  \end{gather}
  \end{remark}

   We mention that     the suggested method based on the Taylor expansion with respect to $t$   can also be used to construct solutions to other equations and
system of equations, which contain derivatives with respect to time for all unknown
functions.



\begin{thebibliography}{99}

\bibitem{Ad.}  Adak D.,  Natarajan E.,  Kumar S.:
A new nonconforming finite-element method for convection  dominated
diffusion-reaction equations,
Int. J. Adv. Eng. Sci. Appl. Math.  $\bf{8}$, 274--283, 2016.

\bibitem{Be.} Besov, O.V., Il'in, V.P., Nikolsky, S.M.: Integral representation
of functions and embedding theorems. Nauka, Moskow, 1975 (in Russian)

\bibitem{Bu.} Burman E.,  Hansbo F.: Edge stabilization for Galerkin approximation of convection-diffusion-reaction problems, Comput. Methods Appl. Mech. Engrg.  $\bf{193}$,
1437--1453, 2004.

\bibitem{Ciar.} Ciarlet, P.: The finite element method for elliptic problems.
North-Holland, Amsterdam, 1978

\bibitem{Du.} Duvaut, G., Lions, J.-L.: Les in\'{e}quations en m\'{e}canique et en physique. Dunod, Paris, 1972

\bibitem{Do.}  Douglas,  J.Jr.,  Russell, T.F.: Numerical methods for convection-dominated diffusion  problems based on combining the method of characteristics with finite element or finite  difference procedures,
SIAM J. Numer. Anal $\bf{19}$,   871--1090, 1982.

\bibitem{Eid.} Eidelman, S.D.: Parabolic systems. Nauka, Moscow, 1964

\bibitem{Fr.} Friedman, A.: Partial differential equations of parabolic type.
Prentice Hall, New York, 1964

\bibitem{Gir.} Girault, V., Raviart, P.-A.: Finite element approximation of the Navier--Stokes
equations.  Lecture Notes in Mathemathics 749,   Springer, Berlin, 1981

\bibitem{Hei.} Heinrichk, J.C., Pepper, D.W.: Intermediate finite element method.
 Taylor and Francis, Philadelphia, 1999

\bibitem{LSU.} Ladyzhenskaya, O.A., Solonnikov, V.A, Uraltseva, N.N.:
Linear and quasilinear  equations of parabolic type. Amer. Math. Soc., Providence, RI, 1968


\bibitem{Lions} Lions, J.-L.:  Quelques m\'ethodes   de
r\'esolution des probl\`emes aux
limites non lin\'eairies. Dunod, Paris, 1969

\bibitem{Liopt} Lions, J.-L.: Contr\^ole optimal de syst\`{e}mes gouvern\'{e}s
par des
 \'{e}quations aux d\'{e}riv\'{e}es partielles. Dunod, Paris, 1968

\bibitem{Lio.1} Lions, J.-L., Magenes, E.: Probl\'{e}mes aux limites non homog\`{e}nes
et applications, Vol.~1. Dunod, Paris,  1968

\bibitem{LiM.1} Lions, J.-L., Magenes E.: Non-homogeneous boundary value problems and
applications, Vol. 2. Springer, Berlin, 1972


\bibitem{Lit.6} Litvinov, W.G.: Optimization in elliptic problems with applications
to mechanics of deformable bodies and fluid mechanics. Birkh\"auser,
Basel,   2000.

\bibitem{Mau.} Maugin, G.A.: Continuum mechanics of electromagnetic solids.
North-Holland, Amsterdam, 1988

\bibitem{Sch.} Schwartz, L.: Analyse math\'ematique, I. Hermann, Paris, 1967

\bibitem{Sol.2} Solonnikov, V.A.: On boundary value problems for linear
parabolic systems of differential equations of general form.
Trudy MIAN SSSR 83 (1965), 3--162  (in Russian). English translation: Proceedings
of the Steklov Institute of Mathematics, no. 83 (1965), Boundary Value Problems
of Mathematical Physics III, edited by O.A. Ladyzhenskaya.

\end{thebibliography}
\end{document}